\documentclass[12pt,a4paper,reqno]{amsart}
\usepackage[T1]{fontenc}
\usepackage{lmodern}
\usepackage{amsmath,amsthm,amssymb,amsfonts,mathtools}
\usepackage[utf8]{inputenc}
\usepackage{tikz}
\usepackage{tikz-cd}
\usetikzlibrary{braids,knots}
\usepackage{enumitem}
\setlist[itemize]{label=\textbullet}

\usepackage[sorting=anyt,style=numeric,backend=biber,giveninits=true,maxnames=10,doi=false,isbn=false]{biblatex}
\DeclareFieldFormat{pages}{#1}
\renewbibmacro{in:}{}

\usepackage{listings}
\usepackage{csquotes}
\usepackage[titletoc]{appendix}
\usepackage[colorlinks,linkcolor=black,urlcolor=blue,citecolor=black]{hyperref}
\usepackage{geometry}
\usepackage{cellspace}
\usepackage{adjustbox}
\usepackage{comment}
\usepackage{algorithm}
\usepackage[noend]{algpseudocode}
\theoremstyle{plain}
\newtheorem{thm}{Theorem}
\numberwithin{thm}{section}
\newtheorem{prop}[thm]{Proposition}
\newtheorem{cor}[thm]{Corollary}
\newtheorem{lem}[thm]{Lemma}

\newtheorem{defi}[thm]{Definition}

\newtheorem{propdef}[thm]{Proposition-Definition}
\newtheorem{ex}[thm]{Example}

\newtheorem{rmk}[thm]{Remark}

\newtheorem*{thm*}{Theorem}
\newtheorem*{conj*}{Conjecture}
\newtheorem*{prop*}{Proposition}
\newtheorem*{ex*}{Example}

\tikzset{
    labl/.style={anchor=south, rotate=90, inner sep=.5mm}
}
\providetoggle{nomsort}
\settoggle{nomsort}{true} 

\newcommand{\Sym}{\mathfrak{S}}

\newcommand{\N}{\mathbb N}
\newcommand{\C}{\mathbb  C}
\newcommand{\Z}{\mathbb Z}
\newcommand{\R}{\mathbb R}
\newcommand{\Tr}{\text{Tr}}

\tikzset{
    labl/.style={anchor=south, rotate=90, inner sep=.5mm}
}
\makeatletter
\def\subsection{\@startsection{subsection}{2}%
  \z@{.5\linespacing\@plus.7\linespacing}{.3\linespacing}%
  {\normalfont\bfseries}}
\makeatother

\title[Hecke algebras for set-theoretical solutions to the YBE]{Hecke algebras for set-theoretical solutions to the Yang--Baxter equation}
\author{Edouard Feingesicht}
\address{Normandie Univ, UNICAEN, CNRS, LMNO, 14000 Caen, France}
\email{edouard.feingesicht@unicaen.fr}
\makeatletter
\@namedef{subjclassname@2020}{\textup{2020} Mathematics Subject Classification}
\makeatother
\begin{document}
\begin{abstract}We define a concept of Hecke algebra for structure groups of set-theoretical solutions to the Yang--Baxter equation. As a comparison to Artin--Tits groups of spherical type, we study some properties of this construction, while also highlighting some differences that appear, which shows a difference between finite Coxeter groups and the "Coxeter-like" group introduced by Dehornoy. We also relate this definition to known constructions on solutions (retractions). Finally, we study a particular case related to Torus Knot groups and Complex Reflexion groups.
\end{abstract}
\maketitle
\setcounter{tocdepth}{1}
\section*{Introduction}\label{intro}
The study of involutive non-degenerate set-theoretical solutions to the Yang--Baxter equation (\cite{drinfeld,etingof}) involves many different algebraic structures: cycle sets (\cite{rump}), braces (\cite{brace}), I-structures (\cite{istruct}), etc. The structure group of a solution (\cite{etingof}), was shown to be a Garside group by Chouraqui (\cite{chouraqui}). In \cite{rcc}, Dehornoy constructed a finite quotient of the structure group, which plays a role similar to finite Coxeter groups for Artin--Tits groups of spherical type (\cite{bourbaki}). The construction of this finite quotient involves the existence of a positive integer associated to each solution, usually denoted $d$ and which we call Dehornoy's class. In this article, we are interested in this finite "Coxeter-like" group, with the aim to understand both how it is similar and how it differs from finite Coxeter groups. 

For an Artin--Tits group of spherical type $A$ generated by $S$ and with Coxeter group $W$, the generic Iwahori--Hecke algebra can be defined as a quotient of the group ring $\Z[q^{\pm 1}][A]$ by the relations $s^2=(q-1)s+q$ for all $s$ in $S$. This algebra has numerous interesting properties: it has dimension $|W|$, the generators are invertible, it is semi-simple under a suitable extension, etc. (\cite{bourbaki,geck}). Here, we develop a theory of Hecke algebra for structure group of involutive non-degenerate set-theoretical solutions to the Yang--Baxter equation. We show that this algebra satisfies properties similar to the ones of Iwahori--Hecke algebras of Artin--Tits groups of spherical type, while also highlighting how and why the definition differs. To do so, we will rely on the fact that the structure group of a solution is a brace (\cite{rump07,brace}), which in particular means that there exists an abelian structure on the structure group $G$ such that $(G,+)\cong (\Z^{X},+)$, and we denote by $x^{[k]}$ the element corresponding to $x+\cdots+x$ in $\Z^X$. We summarize our results from each section in the following: 
\begin{thm*}
Let $(X,r)$ be an involutive non-degenerate set-theoretical solution to the Yang--Baxter equation of size $n$ and Dehornoy's class $d$. Denote $G$ the structure group of $(X,r)$ and $\overline G_2=G/\langle x^{[2d]}\rangle_{x\in X}$ its germ associated to $2d$ (two times Dehornoy's class). For any integral domain $R$, define the following $R[q^{\pm 1}]$-algebra:

$$\mathcal H=\raisebox{.33em}{$R[q^{\pm 1}][G]$}\!\!\left/\middle\langle \left(x^{[d]}\right)^2=(q-1)\cdot x^{[d]}+q, \forall x\in X\right\rangle.$$
Then the followings hold:
\begin{itemize}
\item (Theorem \ref{hecke-dim}) $\mathcal H$ is a free $R[q^{\pm 1}]$-module with basis indexed by $\overline G_2$. In particular, $\mathcal H$ has rank $(2d)^n$.
\item (Corollary \ref{cor:inverse_hecke}) If $T_g$ denotes the generator of $\mathcal H$ associated to an element $g$ of $\overline G_2$, then $T_g$ is invertible.
\item (Theorem \ref{thm:invol}) The anti-involution $R[q^{\pm 1}]\to R[q^{\pm 1}]$ that sends $q$ to $q^{-1}$ extends to a well-defined anti-involution of $\mathcal H$ that sends $T_g$ to $T_g^{-1}$ for any $g\in\overline G_2$.
\item (Corollary \ref{cor:tits_C}) If $R=\C$ then $\C(q)\otimes\mathcal H$ is semi-simple, and there is bijection between the irreducible characters of $\C(q)\otimes \mathcal H$ and the irreducible characters of $\C[\overline G_2]$.
\end{itemize}
\end{thm*}
In this explicit version of our results, we chose the polynomial $P(X)=X^2-(q-1)X-q$ to remind of the generic Iwahori--Hecke algebra of Coxeter groups, but our results hold for any polynomial whose leading and constant coefficients are invertible.

In the first section we introduce the necessary definitions on braces that we will need.

In the second section the define of the Hecke algebra is introduced and we show it has the expected dimension (equal to the cardinal of a Coxeter-like group).

In the third section, we explicitly construct an anti-involution on the Hecke algebras, as is known for the case Iwahori--Hecke algebra for Coxeter groups. 

In the fourth section, we provide an application of Tits' Deformation Theorem, relating Hecke algebras and group rings of Coxeter-like groups over a suitable field extension.

Finally, in the fifth section, we focus on a particular example where the naive definition does work, and relate it to known results about Complex Reflection Groups (a generalization of finite Coxeter groups).

Moreover, in Appendix \ref{sec:hecke_intro}, we explain how our definition of the Hecke algebra arises, and why, in general, the naive definition (adapting the definition for Artin--Tits groups of spherical type) doesn't provide the expected properties.
\section{Preliminaries}
The goal of this section is to provide the basic definitions of the approaches used in this article: cycle sets (\cite{rump}) and braces (\cite{brace}). We also give several technical lemmas that will be used in the construction and the study of the Hecke algebras.
\subsection{Cycle sets}
Our basis object to study non-degenerate involutive set-theoretical solutions to the Yang--Baxter equation are cycle sets, which were introduced by Rump (\cite{rump}).  
\begin{defi}[\cite{rump}] A cycle set is a set $S$ endowed with a binary operation $*\colon S\times S\to S$ such that for all $s$ in $S$ the map $\psi(s)\colon t\mapsto s*t$ is bijective and for all $s,t,u$ in $S$:
	\begin{equation}
	\label{RCL}
	(s*t)*(s*u)=(t*s)*(t*u).
	\end{equation}
	When $S$ is finite of size $n$, $\psi(s)$ can be identified with a permutation in $\mathfrak{S}_n$.
	
	When the diagonal map is the identity (i.e. for all $s\in S$, $s*s=s$), $S$ is called square-free.
\end{defi}
From now, we fix a cycle set $(S,*)$.
	\begin{defi}[\cite{rump}] The group $G_S$ associated with $S$ is defined by the presentation:
	\begin{equation}
	\label{RCG}
	G_S\coloneqq\left\langle S\mid  s(s*t)=t(t*s),\: \forall s\neq t\in S\right\rangle.
	\end{equation}
	Similarly, we define the associated monoid  $M_S$ by the presentation: $$M_S\coloneqq\left\langle S\mid  s(s*t)=t(t*s),\: \forall s\neq t\in S\right\rangle^+.$$	
	They will be called the structure group (resp. monoid) of $S$.
	\end{defi}
	\begin{ex}
Let $S=\{s_1,\dots,s_n\}$, $\sigma=(12\dots n)\in\Sym_n$. The operation $s_i*s_j=s_{\sigma(j)}$ makes $S$ into a cycle set, as for all $s,t$ in $S$ we have $(s*t)*(s*s_j)=s_{\sigma^2(j)}=(t*s)*(t*s_j)$.

The structure group of $S$ then has generators $s_1,\dots,s_n$ and relations $s_is_{\sigma(j)}=s_js_{\sigma(i)}$ (which is trivial for $i=j$).

In particular, for $n=2$ we find $G=\langle s,t\mid s^2=t^2\rangle$.
\end{ex}
	When the context is clear, we will write $G$ (resp. $M$) for $G_S$ (resp. $M_S$).
	
	We also assume $S$ to be finite and fix an enumeration $S=\{s_1,\dots,s_n\}$.
	\begin{rmk} 
	By the definition of $\psi\colon S\to \Sym_n$ we have that $s_i*s_j=s_{\psi(s_i)(j)}$, which we will also write $\psi(s_i)(s_j)$ for simplicity.
	\end{rmk}	
\subsection{Braces}
The structure group of a brace has an extra  "ring-like" structure, which was first introduced by Rump in \cite{rump07} as linear cycle sets. An equivalent definition was then introduced by Ced{\'o}, Jespers and Okni{\'n}ski in \cite{leftbrace} and then in a large survey again by Ced{\'o} in \cite{brace}. We will use their definition of a (left) brace throughout this article.
\begin{defi}[\cite{rump07,brace}]
A brace is a triple $(B,+,\cdot)$ such that $(B,+)$ is an abelian group, $(B,\cdot)$ is a group and for all $a,b,c$ in $B$: $$a(b+c)+a=ab+ac.$$
$(B,+)$ will be called the additive group and $(B,\cdot)$ the multiplicative group of the brace $B$.
\end{defi}
We now fix $B$ a brace.
\begin{rmk} Note that, if 0 is the additive identity and 1 the multiplicative identity, then taking $a=1,b=c=0$ yields $1*(0+0)+1=1*0+1*0$, thus $1=0$.
\end{rmk}
\begin{ex}
If $(G,+)$ is an abelian group then $(G,+,+)$ is a brace, called the trivial brace. 

	Taking $(B,+)=\mathbb Z/2\mathbb Z\times \mathbb Z/2\mathbb Z$ with $(a,b)\cdot(c,d)=\begin{cases}(a+c,b+d),& a+b=0 \text{ mod } 2\\(a+d,b+c),& a+b=1\text{ mod } 2\end{cases}$ can be checked to be a left-brace, and obviously $(0,0)$ is the identity of $(B,\cdot)$.
\end{ex}
\begin{propdef}[\cite{brace}]\label{lmap}
For any $a$ in $B$, the map $\lambda:(B,\cdot)\to \text{Aut}(B,+)$ defined by $\lambda_a(b)=ab-a$ for all $a,b$ in $B$, is a well-defined  morphism.

This also gives $ab=a+\lambda_a(b)$. This will be used everywhere to switch between products and sum of elements.
\end{propdef}
\begin{ex} From the previous example we have respectively $\lambda_g=\text{id}_G$ for all $g$ in $G$, and in $(B,+,\cdot)$ $\lambda((a,b))=\sigma^{a+b}$ where $\sigma$ permutes the two coordinate of $(B,+)$, and obviously $(0,0)$ is the identity of $(B,\cdot)$.
\end{ex}
\begin{lem}[\cite{brace}]\label{prodsum}\label{ab-1}\label{ab-1soc}
For any $a,b$ in $B$ we have:
\begin{enumerate}
\item $\lambda_a\lambda_b=\lambda_{a+\lambda_a(b)}$.
\item $ab^{-1}=-\lambda_{ab^{-1}}(b)+a$
\item If $\lambda_a=\lambda_b$ then $ab^{-1}=a-b$
\end{enumerate}
\end{lem}
\begin{proof}
This first one follows from $gh=g+\lambda_g(h)$.

For the second one, $-\lambda_{ab^{-1}}(b)+a=-ab^{-1}b-ab^{-1}+a=ab^{-1}$.

And then, $\lambda_{ab^{-1}}=\lambda_a\lambda_b^{-1}=\lambda_a\lambda_a^{-1}=\text{id}_B.$
\end{proof}
\begin{lem}\label{ybe}
For any $a,b$ in $B$, we have $a\lambda_a^{-1}(b)=b\lambda_b^{-1}(a)$.

Moreover, $\lambda^{-1}_{\lambda^{-1}_a(b)}\lambda_a^{-1}=\lambda^{-1}_{\lambda^{-1}_b(a)}\lambda_b^{-1}$.
\end{lem}
\begin{proof}
Firstly,
$$a\lambda_a^{-1}(b)=a(a^{-1}b-a^{-1})=b-1+a=b-0+a=b+a=a+b=b\lambda_b^{-1}(a).$$
Then from the fact that $\lambda\colon(B,\cdot)\to\text{Aut}(B,+)$ is a morphism we have that $\lambda_{ab}^{-1}=\lambda_b^{-1}\lambda_a^{-1}$ so $$\lambda^{-1}_{\lambda^{-1}_a(b)}\lambda_a^{-1}=\lambda^{-1}_{a\lambda^{-1}_a(b)}=\lambda^{-1}_{b\lambda^{-1}_b(a)}=\lambda^{-1}_{\lambda^{-1}_b(a)}\lambda_b^{-1}.$$
\end{proof}
The following is implicit in \cite{brace}: 
\begin{lem}[\cite{brace}]
Let $S$ be a subset of a brace $(B,+,\cdot$ such that $\lambda_s(S)\subseteq S$ for any $s$ in $S$. Then $(S,+)$ is a subgroup of $(B,+)$ if and only if it is a subgroup of $(B,\cdot)$.
\end{lem}
\begin{proof}
This follows from the identity $ab=a+\lambda_a(b)$ (or equivalently $a+b=a\lambda_a^{-1}(b)$).
\end{proof}
\begin{defi}[\cite{brace}] Let $(B,+,\cdot)$ be a brace.
\begin{itemize}
\item $S\subseteq B$ is a subbrace if it is a subgroup of both $(B,+)$ and $(B,\cdot)$.
\item $L\subseteq B$ is a left ideal if it is a subgroup of $(B,+)$ and $\lambda_a(I)\subseteq I$ for all $a$ in $B$.
\item $I\subseteq B$ is an ideal if it is a normal subgroup of $(B,\cdot)$ and $\lambda_a(I)\subseteq I$ for all $a$ in $B$.
\end{itemize}
\end{defi}
\begin{prop}[\cite{brace}] Let $(B,+,\cdot)$ be a brace and $I\subseteq B$.
\begin{itemize}
\item $I$ is an ideal $\Rightarrow$ $I$ is a left ideal $\Rightarrow$ $I$ is a subbrace.
\item If $I$ is an ideal then the multiplicative quotient $B/I$ has an induced brace structure $(B/I,+,\cdot)$.
\item $\text{Soc}(B)=\text{Ker}(\lambda)=\{a\in B\mid \forall b\in B, ab=a+b\}$ is an ideal called the Socle of B.
\end{itemize}
\end{prop}
In \cite{etingof}, it was shown that there exists a bijective 1-cocyle $\Pi\colon\Z^S\to G$, i.e.a bijective map such that $\Pi(gh)=\Pi(g)\cdot \lambda_g^{-1}(\Pi(h))$ for any $g,h\in\Z^S$. This so called "I-structure" was shown in \cite{rump07,brace} to be equivalent to the following:
\begin{thm}[\cite{etingof,brace}]\label{istruct}
The structure group $G$ of a finite cycle set $S$ has a brace structure given by $(\Z^S,+,\cdot)$ such that $G$ with the usual multiplication is isomorphic to $(\Z^S,\cdot)$.

Moreover, for any $s,t$ in $S$, we have $\lambda^{-1}_s(t)=\psi(s)(t)$.
\end{thm}
From the I-structure mentioned above we can write any element of $G$ as $g=\sum\limits_{s\in S} g_s s$ where $g_s\in\mathbb Z$. Then for any $h$ in $G$, we have $\lambda_h(g)=\sum_S g_s \lambda_h(s)$ with $\lambda_h(s)$ in $S$.

Because we will work over group rings $R[G]$, we will use Dehornoy's notation from \cite{rcc} to denote $s^{[k]}=ks$ for any $k\in\Z$ (to avoid confusion with the element $k\cdot s\in R[G]$).

In \cite{rcc}, Dehornoy constructed a finite quotient of the structure group, which he calls a "Coxeter-like" group:
\begin{thm}[{\cite{rcc,fein}}]\label{germ} Let $S$ be a finite non-trivial cycle set. Then there exists a positive integer $d$ such that $dS\subset \text{Soc}(G)$.

Moreover, for any positive integer $l$, the quotient $\overline G_l\coloneqq G/\langle (ld)s\rangle$ admits a quotient brace structure given by $((\Z/ld\Z)^S,+,\cdot)$.
\end{thm}
In particular, this means that the bijective 1-cocycle $\Pi\colon\Z^S\to G$ induces a bijective 1-cocycle $\overline \Pi\colon(\Z/ld\Z)^S\to\overline G_l$.

The smallest positive integer satisfying the condition of Theorem \ref{germ} is called the Dehornoy's class of $S$, and the set of positive integers that will satisfy the condition are the multiple of $d$. 

We denote by $T$ the diagonal map of $S$ defined by $T(s)=s*s$. The following Proposition will be useful:
\begin{prop}[\cite{fein}]\label{o(T)} The followings hold:
\begin{enumerate}[label=(\roman*)]
\item Let $o$ be the order of $T$ and $k$ any positive integer. Consider the euclidean division of $k$ by $o$ to write $k=o\cdot q+r$, then we have $ks=sT(s)T^2(s)\dots T^{k-1}(s)=(os)^{q}(rs)$.
\item The order $o$ of $T$ divides $d$. In particular, for any integer $k$ and any $s$ in $S$, we have $\lambda^{-1}_{ks}(s)=T^{k}(s)$ and $kds=\left(sT(s)\dots T^{o-1}(s)\right)^k$.
\end{enumerate}
\end{prop}

As in Coxeter group, we can consider reduced word and state an Exchange lemma (see \cite{lcg} for the case of Coxeter groups):
\begin{rmk}\label{rmk:reduced} Consider a word $w=s_{i_1}\cdots s_{i_k}$ over $S$, and let $g$ be the associated element of $\overline G_l$. Then we can write $g=\sum_S g_s s$ with $0\leq g_s<ld$. Thus the word $w$ is a minimal expression of $g$ iff $k=\sum_S g_s$.

So for any $g$ in $\overline G_l$ we denote $\ell(g)=\sum_S g_s$. Then we say that a word $w$ is reduced if $k=\ell(g)$ when $w$ represents $g\in\overline G_l$.
\end{rmk}
\begin{lem}[Exchange Lemma]\label{exchange}
Let $s$ be in $S$ and $g$ in $\overline G$ . Write $g=\sum\limits_{s\in S} g_ss$ with $0\leq g_s<d$. Then either $sg$ is reduced ($\ell(gs)=\ell(g)+1$) or $g_{s*s}=d-1$ (i.e.$(d-1)(s*s)$ left-divides $g$). Moreover, if it is not reduced, then $sg=\sum\limits_{\substack{t\in S\\t\neq s}} g_{s*t}t$.

Moreover, we can go from one reduced expression to another only using the quadratic relations $s(s*t)=t(t*s)$.
\end{lem}
\begin{proof}
As the given expression of $g$ is reduced, we know $\ell(g)=k$, i.e.$\sum\limits_{s\in S} g_s=k$. Now, by Proposition-Definition \ref{lmap} $sg=s+\lambda_s(g)=s+\sum\limits_{t\in S}g_t\lambda_s(t)$. Reindexing the sum by setting $t=\lambda^{-1}_s(u)=s*u$ for some $u\in S$, we have $g=s+\sum\limits_{u\in S}g_{s*u} u$.

This is reduced if and only if $(sg)_u<d$ for all $u$. Because $g$ is reduced, we have $g_{s*u}<d$, so this $sg$ is reduced if and only if $1+g_{s*s}<d$. Meaning that this is not reduced precisely when $g_{s*s}=d-1$. In this case, then $(sg)_s=d$, and we conclude by $ds=0$.

Moreover, assume we have two reduced expressions as $g=s_{i_1}\cdots s_{i_k}$ and $g=s_{j_1}\cdots s_{j_k}$. Using Proposition-Definition \ref{lmap}, we can rewrite both expressions as $g=\sum\limits_{s\in S} g_s s$ and this is unique by the commutativity of $(\overline G,+)$. This rewriting only involves $st=s+\lambda_s(t)=\lambda_s(t)+s=\lambda_s(t)\lambda^{-1}_{\lambda_s(t)}(s)$ which preserves length. Moreover, by Theorem \ref{istruct}, we have that the quadratic relations $s_1(s_1*s_2)=s_2(s_2*s_1)$ are equivalent to $s_1\lambda^{-1}_{s_1}(s_2)=s_2\lambda^{-1}_{s_2}(s_1)$. Letting $s=s_1$ and $s_2=\lambda_{s}(t)$, we see that $st=\lambda_s(t)\lambda^{-1}_{\lambda_s(t)}(s)$ allows us to go from one reduced expression to the other only with the quadratic relations.
\end{proof}
We conclude the preliminaries by the following technical lemma:
\begin{lem}\label{comp} For any $s,t\in S$ the followings hold: 
\begin{enumerate}[label=(\roman*)]
\item There exists $\rho_s$ with $\ell(\rho_s)=d-1$ such that $s^{[d]}=s\rho_s$. Moreover $\rho_s=(s*s)^{[d-1]}$.
\item $\psi(\rho_s)=\psi(s)^{-1}$
\item $s^{[kd]}=(s\rho_s)^k$
\item $s^{[d]}t=t(t*s)^{[d]}$
\item $\rho_s t=(s*t)\rho_{t*s}$
\item $\rho_{s*t}\rho_s=\rho_{t*s}\rho_t$
\item $(s*t)^{[d]}\rho_s=\rho_s t^{[d]}$
\end{enumerate}
For simplicity we will write $\gamma_s^k=\rho_ss^{[(k-1)d]}=(s*s)^{[kd-1]}$ (giving $s\gamma_s^k=s^{[kd]}$).
\begin{enumerate}[resume,label=\alph*)]
\item $\gamma_s^{k}t=(s*t)\gamma_{t*s}^{k}$
\item $\gamma_{s*t}^{k_1}\gamma_s^{k_2}=\gamma_{t*s}^{k_2}\gamma_t^{k_1}$
\end{enumerate}
In particular, when writing $s^{[kd]}=sg$ we have $g=(s*s)^{[kd-1]}=\rho_ss^{[(k-1)d]}=\rho_s(s\rho_s)^{k-1}$. This implies that, if $s^{[d]}=s_1\dots s_d$ then $(s^{[i]})^{[d]}=s_i\dots s_d s_1\dots s_{d-1}$.

Moreover, as all those equalities are true in $G$, they respect length and also hold in $\overline G_k$.
\end{lem}
\begin{proof}
(i) is follows from Proposition \ref{lmap}: $s^{[d]}=s+(d-1)s=s\lambda_s^{-1}((d-1)s)$.

(ii) follows from $1=\psi(s^{[d]})=\psi(s\rho_s)=\psi(s)\psi(\rho_s)$.

(iii) and (iv) follow from the definition of $d$ as we have: $s^{[kd]}=(kd)s=k(ds)=ds\lambda^{-1}_{ds}(ds)\dots\lambda^{-1}_{(k-1)ds}(ds)=(ds)(ds)\dots(ds)=(ds)^k$, and $s^{[d]}t=ds+\lambda_{ds}(t)=t+ds=t\cdot(d\lambda^{-1}_t(s))=t\cdot d(t*s)=t(t*s)^{[d]}$.

For (v) we have $s\rho_s t=s^{[d]}t=t(t*s)^{[d]}=t(t*s)\rho_{t*s}$, applying $t(t*s)=s(s*t)$ and canceling the $s$ gives the result.

For (vi) we have $\rho_{s*t}\rho_s=\rho_{s*t}+\lambda{\rho_{s*t}}(\rho_s)=\rho_{s*t}+(d-1)\psi^{-1}(s*t)(s*s)=\rho_{s*t}+(d-1)\psi(s*t)(s*s)$, from the cycle set equation, we have $\psi(s*t)(s*s)=\psi(t*s)(t*s)$, thus $\rho_{s*t}\rho_s=\rho_{s*t}+(d-1)\psi(s*t)(s*s)=\rho_{s*t}+(d-1)\psi(t*s)(t*s)=\rho_{s*t}+\rho_{t*s}$. By symmetric, we conclude that this is equal to $\rho_{t*s}\rho_t$.

(vii) comes from (iv) applied on $\rho_t=(t*t)^{[d-1]}$ and $\psi(\rho_t)=\psi(t)^{-1}$.

(viii) is deduced from the previous ones: $\gamma_s^{k}t=\rho_ss^{[kd]}t=\rho_s t (t*s)^{[kd]}=(s*t)\rho_{t*s}(t*s)^{[kd]}=(s*t)\gamma_{t*s}^k$

Similarly for (ix): $\gamma_{s*t}^{k_1}\gamma_s^{k_2}=\rho_{s*t}(s*t)^{[k_1d]} \rho_ss^{[k_2d]}=\rho_{s*t}\rho_s t^{[k_1d]}s^{[k_2d]}=\rho_{t*s}\rho_ts^{[k_2d]}t^{[k_1d]}=\rho_{t*s}(t*s)^{[k_2d]}\rho_tt^{[k_1d]}=\gamma_{t*s}^{k_2}\gamma_t^{k_1}$.
\end{proof}
\section{Defining the Hecke algebra}
We fix a cycle set $(S,*)$ of size $n$, of Dehornoy's class $d$, with structure group $G$ and germ $\overline G_l=G/\langle lds\rangle$ for some positive integer $l$.

Recall that, by Theorem \ref{istruct} we have a set bijection, more precisely a bijective 1-cocycle, $\text{cp}\colon G\to\Z^n$. The inverse of this bijective 1-cocycle is also a bijective 1-cocycle $\text{cp}^{-1}=\Pi\colon\mathbb Z^n\to G$: we have $\Pi(gh)=\Pi(g)\lambda_{\Pi(g)}^{-1}(\Pi(h))$. In particular, if $\psi(\Pi(g))=1$, then $\Pi(gh)=\Pi(g)\Pi(h)$. Moreover, by Theorem \ref{germ} $\Pi$ induces a bijective 1-cocycle $\overline\Pi\colon(\mathbb Z/ld\mathbb Z)^n\to\overline G_l$ 

Let $R$ be a ring, and note that $R[\mathbb Z^n]=R[X_1^{\pm 1},\dots,X_n^{\pm 1}]$ by identifying the generator $e_i=(0,\dots,0,1,0,\dots,0)$ with $X_i$. The set map $\Pi$ extends linearly to $R[X_1^{\pm 1},\dots,X_n^{\pm 1}]\to R[G]$, sending $\sum\limits_i r_i X_1^{i_1}\dots X_n^{i_n}$ to $\sum\limits_i r_i \Pi(s_1,\dots,s_1,\dots,s_n,\dots,s_n)$ for some finite indices $i$ and corresponding integers $i_1,\dots,i_n$ and coefficients $r_i$.

We now proceed to construct the Hecke algebra as hinted before: we pick a polynomial, apply it to $s^{[d]}$ and use the 1-cocycle $\Z^n\to G$ to show that we have a basis by showing that the quotients of the associated group rings by appropriate ideals have the same dimensions.

From now on, fix a polynomial $P\in R[X]$ of degree $l>0$ and set $P(X)=\sum\limits_{k=0}^l a_kX^k$.

\begin{rmk}\label{rmk:ideal}
Recall that given an algebra $A$ and $R\subseteq A$, elements of the two sided ideal generated by $R$ are of the form $\sum a_ir_ib_i$, a finite sum where $a_i,b_i\in A,r_i\in R$.
\end{rmk}
\begin{lem}\label{pialg}
Consider the two-sided ideals $I_P=\left(P(X_1^d),\dots,P(X_n^d)\right)\subset R[\mathbb Z^n]$ and $J_P=\left(P(s_1^{[d]}),\dots,P(s_n^{[d]})\right)\subset R[G]$. Then $\Pi$ induces a bijection $I_P\to J_P$.
\end{lem}
\begin{proof}
First remark that $P$ sends a set of generators of $I_P$ to a set of generators of $J_P$: $$\Pi(P(X_i^d))=\Pi\left(\sum a_k X_i^{kd}\right)=\sum a_k\Pi\left(X_i^{kd}\right)=\sum a_k s_i^{[kd]}=\sum a_k (s_i^{[d]})^k=P(s_i^{[d]})$$ where we use that $S$ is of class $d$ with Proposition \ref{o(T)} to have $s_i^{[kd]}=(s_i^{[d]})^k$.

As $\Pi\colon \Z^n\to G$ is bijective, its linearization $\Pi\colon R[X_1,\dots,X_n]\to R[G]$ is bijective. But $\Pi$ is not a morphism (only a bijective 1-cocycle), so we can't deduce that $\Pi(P(X_i^d))=P(s_i^{[d]})$ to obtain $\Pi(I_P)\subseteq J_P$. However, we'll use that $\Pi$ is a 1-cocycle and $S$ is of class $d$, to deduce that, for any $1\leq i\leq n$ and any $f\in R[\mathbb Z^n]$, we have $\Pi(X_i^df)=\Pi(X_i^d)\cdot \lambda_{\Pi(X_i^d)}^{-1}(\Pi(f))=s_i^{[d]}\Pi(f)$.

We'll prove that $\Pi(I_P)=J_P$ by double inclusion: 

Let $Q_1,Q_2\in R[\mathbb Z^n]$. By the commutativity of $R[\mathbb Z^n]=R[X_1,\dots,X_n]$, we have that $Q_1P(X_i^d)Q_2=P(X_i^d)Q_1Q_2$ for any $1\leq i\leq n$. Moreover, as $S$ is of class $d$ and $\Pi$ is a 1-cocycle, we have $\Pi\left(X_i^d(X_1^{b_1}\dots X_n^{b_n})\right)=s_i^{[d]}\Pi(X_1^{b_1}\dots X_n^{b_n})$. Thus $\Pi(Q_1P(X_i^d)Q_2)=\Pi(P(X_i^d))\Pi(Q_1Q_2)=P(s_i^{[d]})\Pi(Q_1Q_2)$, which is in $J_P$ as $J_P$ is an ideal. So we have $\Pi(I_P)\subseteq J_P$.

Now let $f,g\in G$. Then, by Lemma~\ref{ybe}, we have for all $g\in G$ that $gs^{[d]}=\lambda_g(s^{[d]})\lambda_{s^{[d]}}(g)=\psi(g)(s^{[d]})g$. Thus, in $R[G]$, we have $$fP(s_i^{[d]})g=\sum a_k fs_i^{[dk]}g=\sum a_k (\psi(f)^{-1}(s_i))^{[dk]}fg.$$ Write $t=(\psi(f)^{-1}(s_i))$ and let $Y\in\{X_1,\dots,X_n\}$ be such that $\Pi(Y)=t$. As $S$ is of class $d$, we have $\Pi^{-1}(fP(s_i^{[d]})g)=\sum a_k \Pi^{-1}(fs^{[dk]}g)=\sum a_k Y^{dk}\Pi^{-1}(fg)=P(Y^d)\Pi^{-1}(fg)$, which is in $I_P$ by Remark \ref{rmk:ideal}. We conclude that $J_P\subseteq \Pi(I_P)$.
\end{proof}
\begin{ex}
Let $P(X)=1+X$, $g\in G$ and $s\in S$. Write $\Pi(X)=s,Q=\Pi^{-1}(g)$, $t=\psi(g)^{-1}(s)$ and $Y=\Pi^{-1}(t)$. Then $(1+g)(1+s^{[d]})=1+s^{[d]}+g+gs^{[d]}=1+s^{[d]}+g+t^{[d]}g=(1+s^{[d]})+(1+t^{[d]})g=P(s^{[d]})+P(t^{[d]})g=\Pi(P(X^d))+\Pi(Y^d)\Pi(Q)=\Pi(P(X^d)+P(Y^d)Q)$. 

Thus $(1+g)(1+s^{[d]})$ is an element of $(P(s))\subset J_P$, with preimage $P(X^d)+P(Y^d)Q$ in $(P(X^d),P(Y^d))\subset I_P$.
\end{ex}
The following examples highlight why we need to take polynomials in $X^d$:
\begin{ex}
Let $(S,*)$ be the cycle set, with $S=\{s,t,u\}$ and $\psi(s)=\psi(t)=\psi(u)=(stu)=\sigma$. 
Then $S$ is of class 3 and $s^{[3]}=stu,t^{[3]}=tus,u^{[3]}=ust$. Write $R[\Z^3]=R[X,Y,Z]$ where $\Pi(X)=s,\Pi(Y)=t,\Pi(Z)=u$. Let $P(x)=1+x^2$ and consider the ideals $I=(P(X^2),P(Y^2),P(Z^2))$ and $J=(P(s^{[2]}),P(t^{[2]}),P(u^{[2]}))$. 

Note that, for $T_i\in\{X,Y,Z\}$ with $1\leq i\leq k$, $$\Pi(T_1\cdots T_k)=\Pi(T_1)\cdot \sigma(\Pi(T_2))\cdots \cdot\sigma^{k-1}(\Pi(T_k))$$ as $\Pi$ is a 1-cocycle. Or equivalently, for $t_i\in\{s,t,u\}$, $$\Pi^{-1}(t_1\dots t_k)=\Pi^{-1}(t_1)\Pi^{-1}(\sigma^{-1}(t_2))\cdots \Pi^{-1}(\sigma^{-k+1}(t_k)).$$

Now the element $f=tt+tstt=t(1+st)t\in R[G]$ is in $J$, as $P(s^{[2]})=1+s^{[2]}=1+st$. However, $\Pi^{-1}(tt)=\Pi^{-1}(t)\Pi^{-1}(\sigma^{-1}(t))=\Pi^{-1}(t)\Pi^{-1}(s)=YX$, and similarly $\Pi^{-1}(tstt)=t\Pi^{-1}(u)\Pi^{-1}(u)\Pi^{-1}(t)=YZZY$. Thus $\Pi^{-1}(f)=YX+YZZY=XY+Y^2Z^2$, and we claim that this is not an element of $J$. 

To check that $XY+Y^2Z^2\not\in J$, suppose $XY+Y^2Z^2=a(1+X^2)+b(1+Y^2)+c(1+Z^2)$ with $a,b,c\in R[X,Y,Z]$. As we have no $X^2$ terms, we deduce $a=0$, thus $XY+Y^2Z^2=b(1+Y^2)+c(1+Z^2)$. We have a $XY$ which contains no square term, meaning that $XY$ appears in $b$ or $c$. But there is no $X$ in $Y^2Z^2$, a contradiction.
 
We took polynomials in $X^2$ instead of $X^3$, and now an element of $I$ does not come from $J$. Thus the use of polynomials in $X^d$.

On the other hand , if instead of $st$ we had an element $g$ with trivial permutation (such as $g=stu$), we would have $\Pi^{-1}(t(1+g)t)\in J$. Indeed, $t(1+g)t=t\textcolor{red}{t}+tg\textcolor{blue}{t}$, and as $g$ has trivial permutation, the preimage of the blue $\textcolor{blue}{t}$ would have been the same as the preimage of the red $\textcolor{red}{t}$, allowing for factorization by $\Pi^{-1}(\textcolor{red}{t}\textcolor{blue}{t})$. But with $st$, the blue $\textcolor{blue}{t}$ gets acted on, preventing a factorization. 
\end{ex}
From now on, we fix $P$ in $R[X]$ of degree $l>0$. We furthermore assume that $a_l$, the leading coefficient of $P$, is invertible. We also fix the ideals $I_P\subset R[\Z^n]$ and $J_p\subset R[G]$ as in Lemma \ref{pialg}.

Let $\mathcal H(S,P)=R[G]/J_P$. In $\mathcal H(S,P)$ we thus have that 
\begin{equation}\label{eq:hecke_rel}
\mathcal H(S,P)=R[G]/\left(T_{s^{[ld]}}=\sum_{k=0}^{l-1} \frac{-a_k}{a_l} T_{s^{[kd]}}\right).
\end{equation}

To distinguish between elements of $G$ and their corresponding generator of the algebra, we will write $R[G]=R\langle T_g, g\in G\mid T_gT_h=T_{gh}\rangle$.

\begin{lem}\label{gen} The followings hold:
\begin{enumerate}[label=(\roman*)]
\item We have the isomorphism $R[G]\cong R\langle T_s,s\in S\mid T_sT_{s*t}=T_tT_{t*s},\forall s,t\in S\rangle$
\item For any $\overline{g}\in\overline{G}_l$, there is a well-defined element $T_{\overline{g}}\in\mathcal H(S,P)$ such that $T_{\overline{g}}=T_{s_{i_1}}\cdots T_{s_{i_r}}$ whenever $s_{i_1}\cdots s_{i_r}$ ($s_i\in S$) is a reduced expression of $\overline g$ in $\overline G_l$.
\item For any $g\in G$ with image $\overline g\in\overline G$, if $\ell(g)=\overline\ell(\overline g)$, then the projection $R[G]\to\mathcal H(S,P)$ sends $T_g$ to $T_{\overline g}$.
\end{enumerate}
\end{lem}
\begin{proof}
(i) follows from the definition of the group ring $R[G]$ as the free module with basis $G$ such that $T_gT_h=T_{gh}$ for any $g,h$ in $G$.

For (ii), the Exchange Lemma \ref{exchange} tells us that we can go from one reduced expression to another only using the quadratic relations. By (i) those quadratic relations are also the defining relations of a presentation of $\mathcal H(S,P)$. Thus $T_g$ does not depend on the choice of a reduced expression.

Finally, for (iii), let $g\in G$ and write $g=s_{i_1}\cdots s_{i_r}$ so that $\ell(g)=r$. Let $\overline g$ be the projection of $g$ in $\overline G$, and assume that $\ell(g)=\overline\ell(\overline g)=r$. Then $\overline{s_{i_1}}\cdots\overline{s_{i_r}}$ is a reduced expression of $\overline g$ in $\overline G$. Thus, by (ii), $T_{\overline g}=T_{\overline{s_{i_1}}}\cdots T_{\overline{s_{i_r}}}$ is the projection of $T_g$.
\end{proof}
Recall that,by Lemma \ref{comp}, we have $s^{[ld]}=s\cdot (s*s)^{[ld-1]}$ . Thus, Equation (\ref{eq:hecke_rel}) means that, in $\mathcal H(S,P)$ we have
\begin{equation}
\label{eq:hecke_rel_germ}
T_s T_{s*s}^{[ld-1]}=\sum_{k=0}^{l-1} \frac{-a_k}{a_l} T_{s^{[kd]}}.
\end{equation}
Even though $T_s^{[ld]}$ is not defined in $\mathcal H(S,P)$ from Lemma \ref{gen}, we will often abuse notation and write $T_s^{[ld]}$ instead of $T_s T_{s*s}^{[ld-1]}$ in $\mathcal H(S,P)$. 
\begin{lem}
As an $R$-module, $\mathcal H(S,P)$ is generated by $\{T_g\}_{g\in\overline G_l}$.
\end{lem}
In particular, this means that $\mathcal H(S,P)$ is finite dimensional, and that its dimension is bounded above by $\#\overline G_l=(ld)^n$.
\begin{proof}
Let $s\in S$ and $g\in \overline G_l$. By Remark \ref{rmk:reduced}, either $sg$ is reduced and then $T_sT_g=T_{sg}$, or it is not reduced and $(s*s)^{[ld-1]}\prec g$ ($g=(s*s)^{[ld-1]}h$ is reduced in $\overline G$) by Lemma \ref{exchange}. Thus, if $sg$ is not reduced, by Equation -\ref{eq:hecke_rel_germ}), we have $T_sT_g=T_{s}T_{s*s}^{[ld-1]}T_h=\sum_{k=0}^{l-1} \frac{-a_k}{a_l} T_{s^{[kd]}h}$, where $s^{[kd]}h$ is reduced in $\overline G_l$ as $k<l$ and $g=s^{[ld-1]}h$ is reduced.
\end{proof}
\begin{lem}\label{polyquo}
The quotient algebra $R[\mathbb Z^n]/I$ is a free $R$-module of dimension $(ld)^n$ and basis $X_1^{j_1}\dots X_n^{j_n}$ with $0\leq j_1,\dots,j_n<ld$. 

Moreover, the linearization of $\overline\Pi$ provides a bijection between this basis and $\overline G_l$. 
\end{lem}
The bijection $\Pi$ allows us to write an abuse of notation: by $T_s^{[d]}$ we will mean $T_{s^{[d]}}$.
\begin{proof}
When quotienting $R[X_1,\dots,X_n]$ by $P(X_i^d)$, we can reduce all polynomials of degree strictly greater than $ld-1$. Meaning that $R[\Z^n]/I_P$ has a basis given by $X_1^{j_1}\cdots X_n^{j_n}$ with $0\leq j_i<ld$.

By considering the powers of such a monomial, this basis is in bijection with $(\Z/ld\Z)^n$. By Theorem \ref{germ}, $\overline\Pi$ gives a bijection $(\Z/ld\Z)^n\to\overline G_l$, finishing the proof.
\end{proof}
\begin{thm}\label{hecke-dim}
$\mathcal H(S,P)$ is a free $R$-module with basis $\{T_g\mid g\in\overline G_l\}$, in particular it has dimension $(ld)^n$.
\end{thm}
\begin{proof}
From Lemma \ref{gen} we know that $\{T_g\mid g\in\overline G_l\}$ generates $\mathcal H(S,P)$ as an $R$-module, so in particular $\text{dim} \mathcal H(S,P)\leq (ld)^n$. We just have to show this family is free, but this follows from Lemmas \ref{pialg} and \ref{polyquo}: 

Suppose we have a linear combination $\sum_{\overline g\in\overline G_l} a_{\overline g} T_{\overline g}=0$ in $\mathcal H(S,P)$. By lifting the elements $\overline g\in\overline G_l$ to $g\in G$ we have $\sum_{\overline g\in\overline G_l} a_{\overline g} T_{g}\in J_P$. Then, applying $\Pi^{-1}$ we obtain $\sum_{\overline g\in\overline G_l} a_{\overline g} \Pi^{-1}(T_{g})\in I_P$. Projecting to $R[\mathbb Z^n]/I_P$, this means that $\sum_{\overline g\in\overline G_l} a_{\overline g} \Pi^{-1}(T_{\overline g})=0\in R[\mathbb Z^n]/I_P$. From Lemma \ref{polyquo} the family $\Pi^{-1}(T_{\overline g})$ is a basis of $R[\mathbb Z^n]/I_P$, so we must have $a_{\overline g}=0$ for all $\overline g\in\overline G_l$.
\end{proof}
\begin{rmk}\label{rmk:multi_hecke}
Note that in the above proofs we can take a different polynomial $P$ for each orbit of $S$ under the action of $G$, as in proof of Lemma \ref{pialg} we just need that two elements are in the same orbit to obtain that $\Pi(J)\subseteq I$. If we denote such polynomials by $\underline{P}=(P_i)_{1\leq i\leq n}\in R[X]^n$ with deg $P_i=l_i$ and such that $P_i=P_j$ whenever $s_i$ and $s_j$ are in the same orbit by the action of $\mathcal G$, we obtain the Hecke algebra $\mathcal H(S,\underline{P})$ with dimension equal to $\prod_{i=1}^n (l_id)$ which is the same as the order of the finite group $G/\langle s_i^{[l_id]}\rangle_{1\leq i\leq n}$.

With the same reasoning we can also take a different $d$ for each of those orbits, see for instance \cite{lebed}, but this will not be used in this thesis.

It was chosen to not consider those generalizations (except in Section \ref{ss}) to avoid heavy notation and make the proofs easier to read.
\end{rmk}
\begin{cor}
Taking $P(X)=X^2-pX-q$ with $p,q\in R$ we obtain a definition of an Hecke algebra for cycle sets with  relations of the form $$T_s^{[2d]}=pT_s^{[d]}+q$$
\end{cor}
\begin{ex}\label{ex:hecke_const}
Take $S=\{s_1,\dots,s_n\}$, $\sigma\in \mathfrak S_n$ and $\psi(s_i)=\sigma$. Then $S$ is of class $d=o(\sigma)$ (the order of the permutation), and taking $P(X)=X^2-X-1$ we get $$\mathcal H(S,P)=R\left\langle s_1,\dots,s_n\middle |\begin{array}{cc}
   s_i s_{\sigma(j)}=s_js_{\sigma(i)},&1\leq i<j\leq n\\
   (s_is_{\sigma(i)}\cdots s_{\sigma^{d-1}(i)})^2=s_is_{\sigma(i)}\cdots s_{\sigma^{d-1}(i)}+1,& 1\leq i\leq n
\end{array} \right\rangle$$
\end{ex}
%
\begin{rmk}
For all $g\in G$, by Proposition-Definition \ref{lmap}, we have $\lambda_g(s^{[d]})=(\lambda_g(s))^{[d]}$. So the action of $G$ on $R[G]$ stabilizes $J$ the ideal generated by the $P(s^{[d]})$, meaning that $G$ acts on $\mathcal H(S,P)$. As $ldG\subset\text{Soc}(G)$, the action of $dG$ on $J$ is trivial and thus $\overline G_l$ acts on $\mathcal H(S,P)$.
\end{rmk}
\begin{rmk} We see one important difference between Hecke algebras for Coxeter groups and for Structure group of solutions: for a finite Coxeter group $W$ with associated Artin--Tits group $A$, one can view the Hecke algebra as a deformation of the quotient $R[A]\to R[W]$. However, with our approach for solutions, we have to consider the deformation of a larger quotient $R[G]\to R[\overline G_l]$ with $l>1$ (if $l=1$ then the relations are of the form $T_s^{[d]}=-\frac{a_0}{a_l}$, which is not an interesting deformation).

Moreover, It was shown by Coxeter in \cite{coxeterQuotient} that the quotient $B_n/\langle s^k\rangle$ is finite if and only if $\frac{1}{n}+\frac{1}{k}>\frac{1}{2}$, thus for $n\geq 6$ the quotient is finite only for $k=2$ (the symmetric group). This means that, in the case of Coxeter groups, we can only expect similar definitions of Hecke algebra with polynomials of degree 2.

However here, we can work over any degree, which highlights the different behaviours of the germs and associated Hecke algebra for Coxeter groups and structure groups of solutions. 
\end{rmk}

We conclude the construction of the Hecke algebra for solutions by relating the Hecke algebra of a solution with the Hecke algebra of its retraction. As in Proposition-Definition \ref{pdef:ret}, we denote by $S'$ the retraction of $S$. Then the class $d'$ of $S'$ divides the class $d$ of $S$ by Lemma \ref{lem:ret_div}. We deduce the following:
\begin{prop}
We have a surjective algebra morphism $$\mathcal H(S,P(X))\to\mathcal H\left(S',P\left(X^{\frac{d}{d'}}\right)\right).$$
\end{prop}
\begin{proof}
The morphism $G\to G'$ linearly extends to $R[G]\to R[G']$. By Theorem \ref{o(T)}, for any $s\in S$ and any positive integer $k$, we have $(s^{[d]})^k=s^{[kd]}$. Moreover, by Proposition \ref{lem:ret_div} we know that $d'$ divides $d$, so $(\underline s^{[d']})^{\frac{d}{d'}}=\underline s^{[d]}$. Thus we get $\mathcal H\left(S',P\left(X^{\frac{d}{d'}}\right)\right)=R[G']/\left(P\left((\underline s^{[d']})^{\frac{d}{d'}}\right)\right)=R[G']/\left(P\left(\underline s^{[d]}\right)\right)$. Thus $R[G]\to\mathcal H(S',P(X^{\frac{d}{d'}}))$ factors through $H(S,P)$.
\end{proof}
\begin{ex}
If $S$ is of class $4$, $S'$ of class $2$, and we take $P(X)=X^2+X+1$, then we have a morphism $\mathcal H(S,X^8+X^4+1)\to\mathcal H(S',X^8+X^4+1)$.
\end{ex}

The retraction of a set-theoretical solution to the Yang--Baxter equation was introduced in \cite{etingof}. Here, we relate the Hecke algebra of a solution ahd the Hecke algebra of its retraction. 
\begin{propdef}[\cite{etingof,rump}]\index{Cycle set!Retraction}\label{pdef:ret}
The retraction of $S$ is the quotient set $S'$ by the equivalence relation $s\sim t\Leftrightarrow \psi(s)=\psi(t)$.

The cycle set structure on $S$ naturally induces a cycle set structure on $S'$. Moreover, we also obtain a morphism of cycle sets $S\to S'$, and a morphism of braces $G\to G'$ from the structure brace of $S$ to the one of $S'$.
\end{propdef}
\begin{lem}\label{lem:ret_div}
Let $d$ (resp. $d'$) be the Dehornoy's class of $S$ (resp. $S'$). Then $d'$ divides~$d$.
\end{lem}
\begin{proof}
Let $\underline s$ be the equivalence classes in $S'$ of $s\in S$. Then, from the fact that $G\to G'$ is a morphism of brace and that $S$ is of class $d$, we have in $G'$ $$\lambda_{d\underline{s}}(\underline t)=d\underline{s}\cdot\underline t-d\underline{s}=\underline{ds\cdot t-ds}=\underline{\lambda_{ds}(t)}=\underline t.$$
This means that for all $s$, we have that $d\underline s$ is in the socle of $G_{S'}$. So $d$ is a multiple of $d'$ (the smallest integer such that $dG'\subset\text{Soc}(G')$).
\end{proof}
\begin{ex}
Consider $S=\{s_1,s_2,s_3,s_4\}$ with $\psi(s_1)=\psi(s_3)=(12)(34)$ and $\psi(s_2)=\psi(s_4)=(14)(23)$. Then $S'$ has two elements: $t_1=\{s_1,s_3\}$ and $t_2=\{s_2,s_4\}$, and both $t_1$ and $t_2$ act on $S'$ by the permutation $(12)$. For instance, $t_1*t_2=\underline{s_1}*\underline{s_4}=\underline{s_1*s_4}=\underline{s_3}=t_1$, and this computation does not depend on the choice of representatives for $t_1$ and $t_2$.
\end{ex}
\begin{prop}
We have a surjective algebra morphism $$\mathcal H(S,P(X))\to\mathcal H\left(S',P\left(X^{\frac{d}{d'}}\right)\right).$$
\end{prop}
\begin{proof}
The morphism $G\to G'$ linearly extends to $R[G]\to R[G']$. As $d$ is the Dehornoy's class of $S$, for any $s\in S$ and any positive integer $k$, we have $(s^{[d]})^k=s^{[kd]}$. Moreover, by Proposition \ref{lem:ret_div} we know that $d'$ divides $d$, so $(\underline s^{[d']})^{\frac{d}{d'}}=\underline s^{[d]}$. Thus we get $\mathcal H\left(S',P\left(X^{\frac{d}{d'}}\right)\right)=R[G']/\left(P\left((\underline s^{[d']})^{\frac{d}{d'}}\right)\right)=R[G']/\left(P\left(\underline s^{[d]}\right)\right)$. Thus $R[G]\to\mathcal H(S',P(X^{\frac{d}{d'}}))$ factors through $H(S,P)$.
\end{proof}
\begin{ex}
If $S$ is of class $4$, $S'$ of class $2$, and we take $P(X)=X^2+X+1$, then we have a morphism $\mathcal H(S,X^8+X^4+1)\to\mathcal H(S',X^8+X^4+1)$.
\end{ex}
\section{Anti-involution on the Hecke algebra}
Recall that we fixed a cycle set $(S,*)$ of size $n$, of Dehornoy's class $d$, with structure group $G$ and germ $\overline G_l=G/\langle lds\rangle$. We fix a polynomial $P$ in $R[x]$, written as $P(X)=\sum\limits_{k=0}^l a_kX^k$ with $a_l$ invertible. We have previously defined the Hecke algebra for cycle sets $\mathcal H(S,P)$. In this section, the goal is to endow $\mathcal H(S,P)$ with an anti-involution derived from the inversion in the group $\overline G_l$, in parallel to what is known for finite Coxeter groups (see \cite[Exercise 4.8]{geck} for instance).
\begin{prop}\label{prop:hecke_inv}
Suppose $a_0,a_l$ are invertible in $R$. Then $$T_s^{-1}=\sum_{k=1}^l \frac{-a_k}{a_0}T_{s*s}^{[kd-1]}.$$

Moreover $(T_s^{-1})^{[d]}=(T_{s*s}^{[d]})^{-1}$.
\end{prop}
\begin{proof}
From Lemma \ref{o(T)} we have, for any positive integer $k$, $s^{[k]}=s\cdot (s*s)^{[k]}$. We will use this to check that $\sum\limits_{k=1}^l \frac{-a_k}{a_0}T_{s*s}^{[kd-1]}$ is indeed the inverse of $T_s$:

Firstly, $T_s \left(\sum_{k=1}^l \frac{-a_k}{a_0}T_{s*s}^{[kd-1]}\right)=\sum_{k=1}^{l-1} \frac{-a_k}{a_0}T_s^{[kd]}+\frac{-a_l}{a_0}T_sT_{s*s}^{[ld-1]}$. By Equation (\ref{eq:hecke_rel_germ}) we have $T_s T_{s*s}^{[ld-1]}=\sum_{k=0}^{l-1} \frac{-a_k}{a_l} T_{s^{[kd]}}
$. We conclude that $$T_s \left(\sum_{k=1}^l \frac{-a_k}{a_0}T_{s*s}^{[kd-1]}\right)=\frac{-a_l}{a_0}\frac{-a_0}{a_l}+\sum_{k=1}^{l-1} \left(\frac{-a_k}{a_0}+\frac{a_l}{a_0}\frac{a_k}{a_l}\right)T_s^{[kd]}=1.$$

Then, let $X,Y\in\R[\Z^n]$ be such that $P(X)=s$ and $\Pi(Y)=s*s$. This means that, for $Y'=\sum_{k=1}^l \frac{-a_k}{a_0}Y^{kd-1}$, we have $\Pi(Y')=T_s^{-1}$ and $\Pi(Y'^d)=(T_s^{-1})^{[d]}$. By Lemma \ref{comp}, we have $\psi(\Pi(Y^{kd-1}))=\psi((s*s)^{[kd-1]})=\psi(\rho_s (s*s)^{[(k-1)d]})=\psi(\rho_s)=\psi(s)^{-1}$. Thus, in the sum for $T_s^{-1}$, all the terms  have the same permutation. Now, by Proposition \ref{o(T)} we can write $s^{[d]}=t_1\dots t_d$ where $t_i=t_{i-1}*t_{i-1}$ and $s=t_1=t_d*t_d$ (and so $t_2^{[d]}=t_2\dots t_dt_1$). By Theorem \ref{istruct}, we know that $\Pi$ is a 1-cocycle, meaning that $\Pi(Y'Y')=\Pi(Y')\lambda_{\Pi(Y')}(\Pi(Y'))=\Pi(Y') \psi(s)^{-1}(\Pi(Y'))$. As $\psi(s)^{-1}(s*s)=s$, we have $\psi(s)^{-1}(T_{s*s}^{[kd-1]})=T_s^{[kd-1]}=T_{t_d*t_d}^{[kd-1]}$. Thus $\Pi(Y'Y')=T_{t_1}^{-1}T_{t_d}^{-1}$. By induction, we then have $\Pi(Y'^d)=T_{t_1}^{-1}T_{t_d}^{-1}T_{t_{d-1}}^{-1}\dots T_{t_2}^{-1}=(T_{t_2}\dots T_{t_d} T_{t_1})^{-1}=(T_{t_2}^{[d]})^{-1}$.
\end{proof}
\begin{rmk}\label{rmk:tg-1}
One has to be careful that $T_s^{-1}\neq T_{s^{-1}}$. Indeed, by Lemma \ref{comp} and Proposition \ref{o(T)}, we have $T_{s^{-1}}=T_{\rho_s}=T_{s*s}^{[ld-1]}$, which is only one of the terms occurring in $T_s^{-1}$.
\end{rmk}
\begin{ex}
Take $R=\mathbb Z[q^{\pm 1}]$ and the polynomial $P(X)=X^2-(q-1)X-q=(X-q)(X+1)$, which satisfies the hypotheses of Proposition \ref{prop:hecke_inv}. Then $$T_s^{-1}=\frac{1-q}{q}T_{s*s}^{[d-1]}+\frac{1}{q}T_{s*s}^{[2d-1]}.$$
\end{ex}
\begin{cor}\label{cor:inverse_hecke}
For any $g$ in $\overline G_l$, $T_g$ has an inverse in $\mathcal H(S,P)$.
\end{cor}
\begin{proof}
If $g=t_1\dots t_r$ then the inverse of $T_g=T_{t_1}\cdots T_{t_r}$ is $T_g^{-1}=T_{t_r}^{-1}\cdots T_{t_1}^{-1}$.
\end{proof}
In a group $G$, the map $\iota$ sending an element to its inverse is an anti-involution, that is: $\iota(gh)=\iota(h)\iota(g)$ and $\iota(\iota(g))=g$. This anti-involution is known to extend to the generic Iwahori--Hecke algebra in the case of Coxeter groups \cite[Exercise 4.8]{geck}. We show that the same holds for Hecke algebra of structure groups of solutions to the Yang--Baxter equation, where the algebra is associated to the polynomial $P(X)=\sum\limits_{k=0}^l a_kX^k$ with $a_l$ invertible and $l>0$.
\begin{thm}\label{thm:invol}
If $P$ splits over $R$ with invertible roots $\alpha_1,\dots,\alpha_l$ (not required to be distinct), and if there exists an anti-involution $\iota:R\to R$ sending each $\alpha_i$ to $\alpha_i^{-1}$.

Then $\iota$ extends to an anti-involution of $\mathcal H(S,P)$ by sending $T_g$ to $T_g^{-1}$ for $g$ in $\overline G_l$.
\end{thm}
\begin{proof}
Denote by $\tilde\iota$ the map $\mathcal H(S,P)\to\mathcal H(S,P)$ defined by $\tilde\iota(\sum\limits_{g\in\overline G_l} c_gT_g)=\sum\limits_{g\in\overline G_l} \iota(c_g) T_g^{-1}$. 

We will need that $\iota$ must send $1\in R$ to 1: $\alpha_1^{-1}=\iota(\alpha_1)=\iota(1\cdot\alpha_1)=\iota(\alpha_1)\iota(1)=\alpha_1^{-1}\iota(1)$, thus $\iota(1)=1$. 
 
By the hypothesis that $P$ is split we have 
\begin{equation}\label{eq:hecke_poly_split} P(T_s^{[d]})=0 \Longleftrightarrow a_l\prod_{k=1}^l (T_s^{[d]}-\alpha_k)=0
\end{equation}
For the constant coefficient of $P$ we have $a_0=(-1)^la_l\prod_{k=1}^l \alpha_k$, so $\frac{a_l}{a_0}=(-1)^l\prod_{k=1}^l \alpha_k^{-1}$.

Multiplying Equation (\ref{eq:hecke_poly_split}) by $\frac{a_0}{a_l} \left((T_s^{[d]})^{-1}\right)^l$ yields 
$$a_l\prod_{k=1}^l (-\alpha_k^{-1})(T_s^{[d]})^{-1}(T_s^{[d]}-\alpha_k)=0\Longleftrightarrow a_l\prod_{k=1}^l \left((T_s^{[d]})^{-1}-\alpha_k^{-1}\right)=0.$$

This means precisely that $\tilde\iota(P(T_s^{[d]}))=0$.

Recall from Lemma \ref{comp} the notation $\gamma_s^k=(s*s)^{[kd-1]}$ and that $\gamma_{s*t}^{k_1}\gamma_s^{k_2}=\gamma_{t*s}^{k_2}\gamma_t^{k_1}$. Thus, by Proposition \ref{prop:hecke_inv}  we have$$T_{s*t}^{-1} T_{s}^{-1}=\left(\sum_{k=1}^l \frac{-a_{k}}{a_0}T_{\gamma_{s*t}^{k}}\right)\left(\sum_{k=1}^l \frac{-a_{k}}{a_0}\gamma_{s}^{k}\right)=\left(\sum_{k=1}^l \frac{-a_{k}}{a_0}\gamma_{t*s}^{k}\right)\left(\sum_{k=1}^l \frac{-a_{k}}{a_0}\gamma_{t}^{k}\right)=T_{t*s}^{-1}T_t^{-1}$$ So $\tilde\iota(T_sT_{s*t})=(T_sT_{s*t})^{-1}=T_{s*t}^{-1}T_{s}^{-1}=T_{t*s}^{-1}T_t^{-1}=\tilde\iota(T_tT_{t*s})$.

This shows that $\tilde\iota$ is a well-defined anti-morphism $\mathcal H(S,P)\to\mathcal H(S,P)$.

It remains to show that $\tilde\iota$ is an involution. For this, we will show that $\tilde\iota(\tilde\iota(T_g))$ is an inverse of $\tilde\iota(T_g)=T_g^{-1}$, which will imply that $\tilde\iota(\tilde\iota(T_g))=T_g$. As $\tilde\iota$ is an anti-morphism, we have $\tilde\iota(\tilde\iota(T_g))\tilde\iota(T_g)=\tilde\iota(T_g\tilde\iota(T_g))=\tilde\iota(T_g T_g^{-1})=\tilde\iota(1)=1$. So $\tilde\iota(\tilde\iota(T_g))=T_g$ by unicity of the inverse.
 
Moreover, by Theorem \ref{hecke-dim}, $(T_g)_{g\in\overline G}$ is a basis of $\mathcal H(S,P)$. We conclude that $\tilde\iota$ is an anti-automorphism. Thus $\iota$ is an anti-involution.
\end{proof}
\begin{rmk}
In the above proof, one has to be careful that $T_g^{-1}\neq T_{g^{-1}}$ as mentioned in Remark \ref{rmk:tg-1}. For instance, for the involutivity of $\tilde\iota$, it is not enough to write $\tilde\iota(\tilde\iota(T_g)=\tilde\iota(T_g^{-1})=(T_g^{-1})^{-1}=T_g$. Indeed, for $g=s\in S$, we have $\tilde\iota(T_s^{-1})=\sum_{k=1}^l \iota(\frac{-a_k}{a_0})\tilde\iota(T_{s*s}^{[kd-1]})=\sum_{k=1}^l \iota(\frac{-a_k}{a_0})\tilde\iota(T_{s*s}^{[kd-1]})$, which does not so obviously simplify to $T_s$. 
\end{rmk}
\begin{ex}
Consider $R=\mathbb Z[q_1^{\pm 1},\dots,q_l^{\pm 1},c^{\pm 1}]$. Let $P(X)=c(X-q_1)\dots (X-q_l)$ which satisfies the hypothesis of the theorem. It is an analogue of the "generic Hecke algebra" of a Coxeter group  (\cite{geck}).

Taking as $S=\{s,t\},\psi(s)=\psi(t)=12$ with $P(X)=(X+1)(X-q)=X^2-(q-1)X-q$, we have $T_s^{-1}=\frac{1-q}{q}t^{[1]}+\frac{1}{q}t^{[3]}$. 

We find $(T_s^{-1})^{[2]}=\frac{1}{q}T_t^{[2]}+\frac{1-q}{q}$ and $(T_s^{-1})^{[4]}=\frac{1-q}{q^2}T_t^{[2]}+\frac{q^2-q+1}{q^2}$. 

Thus $$(T_s^{-1})^{[4]}-(\frac{1}{q}-1)(T_s^{-1})^{[2]}-\frac{1}{q}=0$$
\end{ex}
\section{Semi-simplicity}\label{ss}

This section is based on \cite{curtis1,curtis2,geck} and inspired from the lecture notes \cite{digne,michel}. For details on character theory we refer to \cite{curtis1}. In this section we fix a commutative integral domain $R$ with field of fractions $F$, $K$ a field with an algebraic closure $\overline K$, $f\colon R\to K$ a ring morphism. Let $\mathcal H=\mathcal H(S,\underline P)$ be the Hecke algebra of a cycle set $S$, as in Remark \ref{rmk:multi_hecke}, with $\underline P=(P_i)_{1\leq i\leq n}\in R[X]^n$, $P_i(X)=\sum_{i=0}^{l_i} a_{i,k} X^i$ such that $P_i=P_j$ whenever $s_i$ and $s_j$ are in the same $\mathcal G$-orbit. From Theorem \ref{hecke-dim}, This algebra dimension is equal to the order of the quotient group $\overline G_{\underline l}=G/\langle s_i^{[l_id]}\rangle$.

\begin{defi}
Let $A$ be a non-trivial $K$-algebra. Then $A$ is called,
\begin{enumerate}[label=(\roman*)] 
\item simple if it contains no proper two-sided ideal,

\item semi-simple if it is isomorphic to a direct sum of simple algebras,

\item separable if for any extension $L/K$, $L\otimes A$ is a semi-simple algebra,

\item split if it is semi-simple and it is isomorphic to a finite sum of matrix algebras over $K$.
\index{Algebra}
\index{Algebra!Simple}
\index{Algebra!Semi-simple}
\index{Algebra!Separable}
\index{Algebra!Split}
\end{enumerate}
\end{defi}
An ideal $I$ of an algebra is called nilpotent if there exists a positive integer $n$ such that $I^n=0$, i.e. any product of $n$ elements of $I$ is 0. The following proposition helps us characterizing semi-simple algebras:
\begin{prop}[{\cite[Section 9]{bourbaki8}}]\label{prop:rad}\index{Algebra!Radical} Let $A$ be a finite dimensional $K$-algebra. Then there exists a unique largest nilpotent two-sided ideal, called the radical of $A$ and denoted rad$(A)$.

Moreover, the followings hold: 
\begin{enumerate}[label=(\roman*)]
\item rad$(A)$ is the set of elements of $A$ acting as 0 on every simple $A$-module (modules without proper submodules)
\item rad$(A)$ is the intersection of all maximal left ideals of $A$
\item $A$ is semi-simple if and only if rad$(A)=\{0\}$
\end{enumerate}
\end{prop}
In the literature rad$(A)$ is also often called the Jacobson radical of $A$.

If $A$ is a finite-dimensional $K$-algebra and $a$ is an element of $A$, then we denote by $\Tr_{A/K}(a)$ the trace of the left-multiplication operator $A\to A$ defined by $b\mapsto ab$.

If $L/K$ is a field extension, we denote by $A^L$ the $L$-algebra $L\otimes A$.
\begin{lem}[{\cite[]{curtis1}}]\label{lem:tr}
Let $A$ be a finite dimensional $K$-algebra, $L/K$ a field extension. Then for any $a$ in $A$, $\Tr_{A^L/L}(1\otimes a)=\Tr_{A/K}(a)$.

Moreover, $\Tr_{A^L/L}$ is equal to $\text{id}\otimes \Tr_{A/K}$ defined by sending $l\otimes a$ to $l\Tr_{A/K}(a)$.
\end{lem}
\begin{proof}
Let $(e_i)$ be a basis of $A$, so that $(1\otimes e_i)$ is a basis of $A^L$. For $a$ in $A$, write $ae_i=\sum\limits_j c_{ij}e_i$, so that $\Tr_{A/K}(a)=\sum\limits_i c_{ii}$. Then $(1\otimes a)(1\otimes e_i)=1\otimes ae_i=\sum\limits_j c_{ij}(1\otimes e_i)$, meaning that $\Tr_{A^L/L}(1\otimes a)=\sum\limits_i c_{ii}=\Tr_{A/K}(a)$.

Moreover, for any $x\in L$, we then have $(x\otimes a)(1\otimes e_i)=\sum\limits_j c_{ij}(x\otimes e_i)=\sum\limits c_{ij}x(1\otimes e_i)$. Thus $\Tr_{A^L/L}(x\otimes a)=x\sum\limits_i c_{ii}=x\Tr_{A/K}(a)$.
\end{proof}
The following lemma will be useful to restrict to the base field $K$ when studying the trace:
\begin{lem}\label{lem:tr_nd}
Let $A$ be a finite-dimensional $K$-algebra such that the bilinear map $T\colon A\times A\to K$ defined by $T(a,b)=\Tr_{A/K}(ab)$ is non-degenerate. Then for any field extension $L/K$ the bilinear map $T^L\colon A^L\otimes A^L$ defined by $T^L((l_1\otimes a),(l_2\otimes b))=\Tr_{A^L/L}(l_1l_2\otimes ab)$ is non-degenerate.
\end{lem}
\begin{proof}
Let $l\otimes a\in L\otimes A$. As $T$ is non-degenerate, there exists $b\in A$ such that $T(a,b)\neq 0$. Then, by Lemma \ref{lem:tr}, we have $T^L((l\otimes a),(1\otimes b))=T^L(l\otimes ab)=l\text{Tr}_{A/K}(ab)\neq 0$. Thus $T^L$ is non-degenerate.
\end{proof}
\begin{prop}[{\cite[Exercice 7.6]{curtis1}}]\index{Algebra!Trace}
\label{trace}
Let $A$ be a finite dimensional $K$-algebra. If the bilinear form $T\colon A\times A\to K$ defined with the usual trace $T(a,b)=\Tr_{A/K}(ab)$ is non-degenerate, then $A$ is separable (and thus semi-simple).
\end{prop}
\begin{proof}
First we know that non-degeneracy is stable by field extension by Lemma \ref{lem:tr_nd}.

So it is enough to show that $A$ is semi-simple. As $A$ is finite dimensional, by Proposition \ref{prop:rad} $A$ is semi-simple iff rad($A$) is trivial. Also from Proposition \ref{prop:rad}, rad($A$) is the largest nilpotent ideal, so any element in it has trivial trace (any element is nilpotent). Thus as rad($A$) is an ideal, if $a\in\text{rad} A$ then, for any $b\in A$ we have $ab\in\text{rad}(A)$ and so $\Tr_{A/K}(ab)=0$. If $T$ is non-degenerate, this implies that $a=0$, finishing the proof.
\end{proof}
\begin{defi} A trace over a $K$-algebra $A$ is a map $\tau\colon H\to K$ such that $\tau(ab)=\tau(ba)$ for any $a,b$ in $K$. A trace $\tau$ is said to be symmetrizing if the map $(a,b)\mapsto\tau(ab)$ is non-degenerate.
\end{defi}
The following statement is a generalization of Lemma \ref{lem:tr_nd}:
\begin{prop}[{\cite[Proposition 8.7]{broueHecke}}]\label{prop:tr_char} If $A$ is a finite-dimensional algebra over a field $K$ and if $\tau$ is a symmetrizing trace over $A$ that is a linear combination of characters, then $A$ is separable.

In particular, if $\Tr_{A/K}$ is symmetrizing, then $A$ is separable.
\end{prop}
\begin{cor}[{\cite[Exemple 2.10]{digne}}]\label{ringalg-sep} Let $G$ be a group and $K$ a field such that char($K$) does not divide $|G|$. Then the map $\tau\colon K[G]\to K$ defined by $\tau(\sum\limits_{g\in G} r_gT_g)=r_1$ (where $T_g$ is the standard basis of $K[G]$) is a symmetrizing trace and $K[G]$ is separable.
\end{cor}
\begin{proof}
We have that $\tau(\sum\limits_{g\in G} r_g T_g)(\sum\limits_{h\in G} r'_h T_h)=\tau(\sum\limits_{g,h\in G} r_gr'_h T_gT_h)=\sum\limits_{g\in G} r_g r'_{g^{-1}}=\sum\limits_{h\in G} r'_h r_{h^{-1}}$, so $\tau(ab)=\tau(ba)$. Moreover, $\tau(T_gT_g^{-1})=\tau(T_1)=1$, so $\tau((\sum\limits_{g\in G} r_g T_g)T_h^{-1})=r_h$ is zero for every $h$ if and only if $r_h=0$ for every $h\in G$. Thus $\tau$ is non-degenerate, and so it is indeed a symmetrizing trace.

Then, the trace of the algebra $K[G]$ is given on the basis $(T_g)$ by $$\text{Tr}_{K[G]/K}(T_h\mapsto T_{gh})=\#\{h\mid T_{gh}=T_h\}=\begin{cases}\# G,\text{ if } g=1\\0,\text{  otherwise}\end{cases}=\# G\cdot\tau(T_g).$$ Thus $\text{Tr}_{K[G]/K}=\# G\tau$, which is not zero as char $K$ does not divide $\# G$. So $\tau=\frac{\text{Tr}_{K[G]/K}}{\# G}$ is a linear combination of character and finally, by Proposition \ref{prop:tr_char}, $K[G]$ is separable.
\end{proof}
Our goal is to be able to apply the following theorem:
\begin{thm}[{\cite[Tits Deformation Theorem 68.17]{curtis2}}]
\label{tits}
Let $A$ be a finite dimensional $R$-algebra, recall that we chose $F=\text{Frac}(R)$ and $f:R\to K$. If $K\otimes_R A$ and $F\otimes_R A$ (defined by $f$) are separable, then they have the same numerical invariants.

Moreover, let $\overline R$ be an integral closure of $R$ in $\overline K$ and $\overline f:\overline R\to \overline K$ be an extension of $f$. Then $\overline f$ induces a bijection of irreducible characters $\text{Irr}(\overline K\otimes A)\to\text{Irr}(\overline F\otimes A)$. 
\end{thm}
\begin{thm}\label{thm:invariants}
Let $K$ be a field of characteristic $p$. Suppose that $p$ does not divide  $d$, and $p$ does not divide $l_i$ for any $i$ (the degrees of each polynomial). Let $\underline q=(q_{i,k})_{1\leq i\leq n, 0\leq k\leq l_i}$ be a family of indeterminates such that $q_{i,k}=q_{j,k}$ whenever $s_i$ and $s_j$ are in the same orbit and $P_i(X)=\sum\limits_i a_{i,k} X^k\in K[\underline q][X]$. 
Then $K(\underline q)\otimes \mathcal H(S,\underline P)$ is separable and has the same numerical invariants as $K[\overline G_{\underline l}]$. 
\end{thm}
\begin{proof}
Consider the context of Theorem \ref{tits} with $A=\mathcal H(S,\underline P)$, $R=K[\underline q]$, $F=\text{Frac}(R)=K(\underline q)$. We define $f:R\to K$ by $f(q_{i,0})=f(q_{i,l_i})=1$ and otherwise $f(q_{i,k})=0$, so that the specialization given by $f$ yields the algebra $K[\overline G_{\underline l}]=K[G]/(T_{s_i}^{[l_id]}-1)$. 

First, by Corollary \ref{ringalg-sep}, $K\otimes A=K[G_{\underline l}]$ is separable when $\text{char}(K)$ does not divide $\#\overline G_{\underline l}=\prod\limits_{i=1}^n (l_id)
$.

Then, as $R$ is an integral domain, $F=\text{Frac}(R)$ is a field, so
$F\otimes A=K(\underline q)\otimes \mathcal H(S,\underline P)$. We want to show that $\Tr_{F\otimes A/F}$ is symmetrizing, so that we can apply \ref{prop:tr_char} to have that $F\otimes A$ is separable. By Theorem \ref{hecke-dim}, $(T_g)_{g\in\overline G_l}$ is a basis of $A=\mathcal H(S,\underline P)$. So $(1\otimes T_g)$ is a basis of $F\otimes A$. Moreover, $\Tr_{F\otimes A/F}$ specializes to $\Tr_{K[\overline G_{\underline l}]/K}$, which is symmetrizing by Corollary \ref{ringalg-sep}. We have $\Tr_{F\otimes A/F}((1\otimes T_g)(1\otimes T_h))=\Tr_{F\otimes A/F}(1\otimes T_gT_h)=1\otimes\Tr_{A/K}(T_gT_h)$ by Lemma \ref{lem:tr}. As $\Tr_{A/K}$ specializes to $\Tr_{K[\overline G_{\underline l}]/K}$ which is non-degenerate, $\Tr_{F\otimes A/F}$ is also non-degenerate and thus symmetrizing. 

The conditions of Theorem \ref{tits} are satisfied, meaning that $F\otimes A=K(\underline q)\otimes\mathcal H(S,\underline P)$ and $K\otimes A=K[\overline G_{\underline l}]$ have the same numerical invariants.
\end{proof}
\begin{cor}\label{cor:tits_C}
If $\mathcal H(S,\underline P)$ is defined over $\C[\underline q]$, then $\C(\underline q)\otimes \mathcal H(S,\underline P)$ and $\C[\overline G_{\underline l}]$ have the same numerical invariants.

Moreover, we have a bijection $\text{Irr}(\C[\overline G_{\underline l}])\to\text{Irr}\left(\overline{\C(\underline l})\otimes\mathcal H(S,\underline P)\right)$.
\end{cor}
\begin{proof}
We apply Theorem \ref{tits} with: $R=\C[\underline q]$, $A=\mathcal H(S,\underline P)$, $F=\C(\underline q)$, $K=\C=\overline K$ and $K\otimes A=\C[\overline G_{\underline l}]$. Theorem \ref{thm:invariants} already tells us that $\C(\underline q)\otimes \mathcal H(S,\underline P)$ and $\C[\overline G_{\underline l}]$ have the same numerical invariants. Moreover, as $K=\C=\overline K$, the last part of Theorem \ref{tits} says that the specialization $\mathcal H(S,\underline P)\to\C[\overline G_{\underline l}]$ induces a bijection $\text{Irr}(\C[\overline G_{\underline l}])\to\text{Irr}(\overline{\C(\underline l})\otimes\mathcal H(S,\underline P))$.
\end{proof}
\section{Two-generated Cyclic group}\label{ss:2gen}
At the beginning of Section \ref{sec:hecke_intro}, we mentioned how the naive definition of a Hecke algebra does not work in general, and we developed a different approach that provides the expected results. However, we also mentioned that the naive approach does work for a very particular solution of size. For this solution, the structure group is $\langle a,b\mid a^2=b^2\rangle$ and the germ is $\langle a,b\mid a^2=b^2, ab=ba=1\rangle\simeq \mathbb Z/4\mathbb Z$ with algebra $R\langle T_a,T_b\mid T_a^2=T_b^2, T_aT_b=T_bT_a=p(T_a+T_b)+q\rangle$ with some $p,q$ in $R$. The goal of this section is to prove that in this particular case, the Hecke algebra has a basis indexed by the germ. 

Moreover, we study a family of groups for which this approach works: torus knot group, which are the only knot groups (fundamental groups of complements of knots in the 3-sphere) which are Garside groups (\cite{gaussian,gobetTorus,gobetTorusGarside}). For $n$ and $m$ integers strictly greater than 1, the $n,m$-torus knot monoid (resp. group) is defined by the presentation $\mathcal T_{n,m}=\langle a,b\mid a^n=b^m\rangle$, and has as a Garside element $\Delta=a^n=b^m$.

The goal of this section is to show that $\mathcal T_{n,m}$ has a Garside germ given by $\overline{\mathcal  T_{n,m}}=\mathcal T_{n,m}/\langle ab=ba=1\rangle\simeq \mathbb Z/(n+m)\mathbb Z$, and show that we have a Hecke algebra $\mathcal H_{n,m}(p,q)=R\langle T_a,T_b\mid T_a^n=T_b^m, T_aT_b=T_bT_a=p(T_a+T_b)+q\rangle$, i.e. that $(T_g)_{g\in\overline{\mathcal T_{n,m}}}$ is a basis of $\mathcal H_{n,m}(p,q)$.


\begin{prop}\label{prop:tor_germ} $\mathcal T_{n,m}$ is a Garside group with germ $\overline{\mathcal T_{n,m}}\cong\Z/(n+m)\Z$.
\end{prop}
\begin{proof}
It is shown in \cite[Example 4]{gaussian} that, with the given presentation, $\mathcal T_{n,m}$ is a Garside group, with a Garside element $\Delta=a^n=b^m$ and $$\text{Div}(\Delta)=\{1,a,\dots,a^n=b^m,b^{m-1},b^{m-2},\dots,b\}.$$ The additive length $\ell\colon\mathcal T_{n,m}\to \N$ can be obtained by setting $\ell(a)=m,\ell(b)=n$, so that $\ell(a^n)=nm=\ell(b^m)$.

On the other hand,
\begin{flalign*}
\overline{\mathcal T_{n,m}}&\simeq \langle \overline a,\overline b\mid \overline a^n=\overline b^m,\overline a\overline b=\overline b\overline a=1\rangle \simeq \langle \overline a,\overline b\mid \overline a^n=\overline b^m,\overline a=\overline b^{-1}\rangle\simeq \langle \overline a\mid \overline a^n=\overline a^{-m}\rangle&\\
& \simeq \Z/(n+m)\Z=\{1,\overline a,\dots,\overline a^n=\overline b^m,\overline b^{m-1},\overline b^{m-2},\dots,\overline b\}.&
\end{flalign*}
Thus we have a bijection Div$(\Delta)\to\overline{\mathcal T_{n,m}}$ sending $a$ (resp. $b$) to $\overline a$ (resp. $\overline b$).

Let $\overline\ell$ be the induced map of $\ell$ in $\overline{\mathcal T_{n,m}}$, i.e. $\overline\ell(\overline a)=m,\overline\ell(\overline b)=n$.

To show that $\overline{\mathcal T_{n,m}}$ is a Garside germ of $\mathcal T_{n,m}$, we need to show that $$\mathcal T_{n,m}\cong \langle \overline{\mathcal T_{n,m}}\mid \forall g,h\in\overline{\mathcal T_{n,m}}, g\cdot h=gh \text{ when } \overline\ell(gh)=\overline\ell(g)+\overline\ell(h)\rangle.$$
We will prove the isomorphism by showing that the presentation on the right reduces to the presentation of  $\mathcal T_{n,m}$ as $\langle a,b\mid a^n=b^m\rangle$.

As $\{\overline a,\overline b\}\subset\overline{\mathcal T_{n,m}}$, $\overline{\mathcal T_{n,m}}$ generates $\mathcal T_{n,m}$. Now for the relations, we have to consider the products $\overline a^i\overline a^j,\overline b^i\overline b^j$ and $\overline a^i\overline b^j$:

We have $\overline\ell(\overline a^i)+\overline\ell(\overline a^j)=im+jm=(i+j)m$ for $1\leq i,j\leq n$. If $i+j\leq n$, then $\overline\ell(\overline a^i\overline a^j)=\overline\ell(\overline a^{i+j})=(i+j)m$. Thus we can omit $\overline a^i$ for $2\leq i\leq n$ from the generators. The same holds for $\overline b^j$, as $\overline\ell(\overline b^i)+\overline\ell(\overline b^j)=in+jn=(i+j)n=\overline\ell(\overline b^{i+j})$, if $i+j\leq m$. Thus we can omit $\overline b^i$ for $2\leq i\leq m$ from the generators. The particular case of $\overline b\overline b^{m-1}=\overline b^m=\overline a^n$ with $\overline\ell(\overline b^m)=nm=\overline\ell(\overline a^n)$ recovers the relation $\overline a^n=\overline b^m$.

However, the longest length in $\overline{\mathcal T_{n,m}}$ is $\overline\ell(\overline a^n)=\overline\ell(\overline b^m)=nm$. So if $i+j>n$, $\overline\ell(\overline a^i)+\overline\ell(\overline a^j)=in+jn=(i+j)m>nm$, so there is no relation for this case. The same also holds for $\overline b$ whenever $i+j>m$.

Finally, $\overline\ell(\overline a^i)+\overline\ell(\overline b^j)=im+jn$ for $1\leq i\leq n, 1\leq j\leq m$. But $\overline a=\overline b^{-1}$, so $\overline\ell(\overline a^i\overline b^j)=\begin{cases}\overline\ell(\overline a^{i-j})=(i-j)m,\text{ if } i\geq j\\\overline\ell(\overline b^{j-i})=(j-i)n,\text{ if } i<j\end{cases}$. In both cases this is smaller than $im+jn$, so there is no relation.

From this, we conclude that the only relation left that occurs from $\overline{\mathcal T_{n,m}}$ is $
\overline a^n=\overline b^m$, showing the desired result.
\end{proof}
Now consider $\mathcal H_{n,m}(p,q)=R\langle T_a,T_b\mid T_a^n=T_b^m, T_aT_b=T_bT_a=p(T_a+T_b)+q\rangle$ for some $p,q$ in $R$.
\begin{lem} The followings hold: 
\begin{enumerate}[label=(\roman*)]
\item In $\mathcal H_{n,m}(p,q)$ we have,
\begin{enumerate}
\item $T_aT_b^k=p^{k-1}q+p^kT_a+\sum\limits_{i=1}^{k-1} (p^2+q)p^{k-i-1}T_b^i+pT_b^k$, for $1\leq k\leq m$ 
\item $T_bT_a^k=p^{k-1}q+p^kT_b+\sum\limits_{i=1}^{k-1} (p^2+q)p^{k-i-1}T_a^i+pT_a^k$, for $1\leq k\leq n$
\end{enumerate}
\item $(T_g)_{g\in\overline{\mathcal T_{n,m}}}$ generates $\mathcal H_{n,m}(p,q)$
\end{enumerate}
\end{lem}
\begin{proof}
For (i) we proceed by induction on $k$. If $k=1$, then $T_aT_b=q+pT_a+pT_b=p^{1-1}q+p^1T_a+pT_b^1$ (and the sum is empty). Now assume the equality holds for some $1\leq k<m$, then we have $T_aT_b^{k+1}=(T_aT_b^{k})T_b=(p^{k-1}q+p^kT_a+\sum\limits_{i=1}^{k-1} (p^2+q)p^{k-i-1}T_b^i+pT_b^k)T_b=p^{k-1}qT_b+p^kT_aT_b+\sum\limits_{i=1}^{k-1} (p^2+q)p^{k-i-1}T_b^{i+1}+pT_b^{k+1}$. We have $p^kT_aT_b=p^k(pT_a+pT_b+q)=p^{k+1}T_a+p^{k+1}T_b+p^{k}q$ and we can rewrite $\sum\limits_{i=1}^{k-1} (p^2+q)p^{k-i-1}T_b^{i+1}=\sum\limits_{i=2}^{k} (p^2+q)p^{(k+1)-i-1}T_b^i$. Thus, rearranging the terms, we obtain $T_aT_b^{k+1}=p^{k}q+p^{k+1}T_a+p^{k-1}qT_b+p^{k+1}T_b+\sum\limits_{i=2}^{k} (p^2+q)p^{(k+1)-1-i}T_b^i+pT_b^{k+1}=p^{k}q+p^{k+1}T_a+\sum\limits_{i=1}^{k} (p^2+q)p^{(k+1)-i-1}T_b^i+pT_b^{k+1}$.
A totally symmetric argument holds for $T_bT_a^k$.

For (ii), we can use that $T_a^{n+1}=T_aT_a^n=T_aT_b^m$ (resp. $T_b^{m+1}=T_bT_b^m=T_bT_a^n$) and the apply the relations of (1) to reduce terms of high enough exponents. Thus, with the relations of (1), any product of generators can be reduced to linear combinations of the family $(T_g)_{g\in\overline{\mathcal T_{n,m}}}$.
\end{proof}
\begin{thm}
The family $(T_g)_{g\in\overline{\mathcal T_{n,m}}}$ is a basis of $\mathcal H_{n,m}(p,q)$. In particular $\mathcal H_{n,m}(p,q)$ has dimension $n+m$.
\end{thm}
The proof will follow a common strategy for Hecke algebra of finite Coxeter groups, see \cite[Theorem 4.4.6]{geck}.
\begin{proof}
Consider $E$ the free $R$-module with basis $(e_g)_{g\in\overline{\mathcal T_{n,m}}}$. We are going to show that we have an action of $\mathcal H_{n,m}(p,q)$ over $E$ induced by $T_ge_1=e_g$ and this will be enough. Indeed, assuming we have a linear combination $\sum_{g\in\overline{\mathcal T_{n,m}}} r_gT_g=0$ then $0=(\sum_{g\in\overline{\mathcal T_{n,m}}} r_gT_g)e_1=\sum_{g\in\overline{\mathcal T_{n,m}}} r_ge_g$ and since $E$ is free over $(e_g)$ we deduce that $r_g=0$ for all $g$.

We define the following action of $\overline{\mathcal T_{n,m}}$ on $E$, and show that it induces an action of $\mathcal H_{n,m}(p,q)$ on $E$: 
\begin{itemize}
\item $T_ae_{a^k}=e_{a^{k+1}}$, for $0\leq k\leq n-1$
\item $T_ae_{b^k}=p^{k-1}qe_1+p^ke_a+\sum\limits_{i=1}^{k-1} (p^2+q)p^{k-i-1}e_{b^i}+pe_{b^k}$, for $1\leq k\leq m$
\item $T_be_{b^k}=e_{b^{k+1}}$, for $0\leq k\leq m-1$
\item $T_be_{a^k}=p^{k-1}qe_1+p^ke_b+\sum\limits_{i=1}^{k-1} (p^2+q)p^{k-i-1}e_{a^i}+pe_{a^k}$, for $1\leq k\leq n$
\end{itemize}
In particular, $T_ae_{a^n}=T_ae_{b^m}=p^{m-1}qe_1+p^me_a+\sum\limits_{i=1}^{m-1} (p^2+q)p^{m-i-1}e_{b^i}+pe_{b^m}$.


We will to show that this action respect the defining relations of $\mathcal H_{n,m}(p,q)$.

To verify that the action is compatible with the relation $T_aT_b=p(T_a+T_b)+q$, we only need to consider the cases of $T_aT_be_{b^k}$ and $T_aT_be_{a^k}$, as the cases of $T_bT_ae_{b^k}$ and $T_bT_ae_{a^k}$ are obtained by symmetry.
First assume that $k<m$, then, on one hand, $T_aT_be_{b^k}=T_ae_{b^{k+1}}=p^{k}qe_1+p^{k+1}e_a+\sum\limits_{i=1}^{k} (p^2+q)p^{k-i}e_{b^i}+pe_{b^{k+1}}$. On the hand, $(pT_a+pT_b+q)e_{b^k}=pT_ae_{b^k}+qe_{b^k}+pe_{b^{k+1}}=p^{k}qe_1+p^{k+1}e_a+\sum\limits_{i=1}^{k-1} (p^2+q)p^{k-i}e_{b^i}+p^2e_{b^k}+qe_{b^k}+pe_{b^{k+1}}$ and those are easily seen to be equal by just noticing $p^2e_{b^k}+qe_{b^k}=(p^2+q)p^{k-k}e_{b^k}$. 

Then, for $k<n$, we have $T_aT_be_{a^k}=T_a(p^{k-1}qe_1+p^ke_b+\sum\limits_{i=1}^{k-1} (p^2+q)p^{k-i-1}e_{a^i}+pe_{a^k})=p^{k-1}qe_a+p^kT_ae_b+\sum\limits_{i=1}^{k-1} (p^2+q)p^{k-i-1}e_{a^{i+1}}+pe_{a^{k+1}}$ and a bit of rearranging the terms (and changing indices of sum) show that this is equal to $T_be_{a^{k+1}}=T_bT_ae_{a^k}$ which, again by symmetry, finishes the case $k<n$. 

Now for $k=m$ we have $T_aT_be_{b^m}=T_aT_be_{a^n}=T_a(p^{n-1}qe_1+p^ne_b+\sum\limits_{i=1}^{n-1} (p^2+q)p^{n-i-1}e_{a^i}+pe_{a^n})$, so 
\begin{equation}\label{eq:abbm}
T_aT_be_{b^m}=p^{n-1}qe_a+p^nT_ae_b+\sum\limits_{i=1}^{n-1} (p^2+q)p^{n-i-1}e_{a^{i+1}}+pT_ae_{a^n}.
\end{equation}
On the other hand, 
\begin{equation}\label{eq:abbm2}
(pT_a+pT_b+q)e_{b^m}=pT_ae_{b^m}+pT_be_{b^m}+qe_{b^m}. 
\end{equation}
The last term of Equation (\ref{eq:abbm}) and the first term of Equation (\ref{eq:abbm2}) match, as $a^n=b^m$. So we have to show $$p^{n-1}qe_a+p^nT_ae_b+\sum\limits_{i=1}^{n-1} (p^2+q)p^{n-i-1}e_{a^{i+1}}=pT_be_{b^m}+qe_{b^n}.$$  On the left we expand $T_ae_b$ and on the right we expand $T_be_{b^n}=T_be_{a^n}$, where we respectively obtain 
$$p^nqe_1+p^{n+1}e_b+\sum\limits_{i=1}^{n} (p^2+q)p^{(n+1)-i-1}e_{a^i}$$
and 
$$p^{n}qe_1+p^{n+1}e_b+\sum\limits_{i=1}^{n-1} (p^2+q)p^{(n+1)-i-1}e_{a^i}+p^2e_{a^n}+qe_{a^n}$$
which also match as $(p^2+q)e_{a^n}=(p^2+q)p^{(n+1)-n-1}e_a^{n}$.

For $T_bT_ae_{b^m}$ the computation is totally similar.

Then we can easily deduce that the relation $T_a^n=T_b^m$ is compatible with the action: $$T_a^ne_{a^k}=T_a^ke_{a^n}=T_a^ke_{b^m}=T_a^kT_b^me_1=T_b^mT_a^ke_1=T_b^me_{a^k}$$ 
The first equality is obtained by $T_ae_{a^k}=e_{a^{k+1}}$ for $k<n$. The second one by $a^n=b^m$. The third equality is obtained by $T_be_{b^k}=e_{b^{k+1}}$ for $k<m$. The fourth one follows from the fact that we've shown that $T_aT_b=T_bT_a$ is respected by the action. 

Similarly, we have $$T_a^ne_{b^k}=T_a^nT_b^ke_1=T_b^kT_a^ne_1=T_b^ke_{a^n}=T_b^ke_{b^m}=T_b^me_{b^k}.$$
Showing that the action of $\mathcal H_{n,m}(p,q)$ on $E$ is well-defined, and thus finishing the proof.
\end{proof}
We finish this section by relating this result with a well-known theory for Complex reflection groups (CRG), following \cite{CRGHecke}:\index{Complex reflection group}
\begin{defi} Let $V$ be a complex vector space of finite dimension $r$.

A pseudo-reflection is a non-trivial element of $\text{GL}(V)$ that fixes an hyperplane in $V$.

A complex reflection group of rank $r$ is a finite subgroup of $\text{GL}(V)$ generated by pseudo-reflections. Moreover, a complex reflection group is called irreducible if it does not stabilize any proper subspace of $V$.
\end{defi}
The classification of all irreducible complex reflection groups was obtained by Shephard and Todd in \cite{shephardCRG}, involving an infinite family $G(de,e,r)$ with $d,e,r$ positive integers, and 34 exceptional cases $G_4,G_5,\dots,G_{37}$. Moreover, the family of complex reflection groups whose elements are real matrices correspond to finite Coxeter groups. Thus, they are often seen as a natural generalization of finite Coxeter groups.

In \cite{CRGHecke}, the authors give a topological definition of the Hecke algebra of a CRG. The authors then show that for the infinite family $G(de,e,r)$, the Hecke algebra admits a presentation with generators $T_s$ associated to the pseudo-reflections generating the CRG, and relations of two types: "braid-like" relations, and relations of the form $(T_s-u_{s,0})(T_s-u_{s,1})\cdots(T_s-u_{s,e_s})$ for some integer $e_s$.

As in the section we focused on the Garside group $\mathcal T_{n,m}$ of rank 2 with germ $\overline{\mathcal T_{n,m}}\cong \Z/(n+m)\Z$, we provide the statement of \cite{CRGHecke} for the case $G(k,1,1)\cong \Z/k\Z$:
\begin{thm}[{\cite[Propositions 4.22-4.24]{CRGHecke}}]
For the Hecke algebra of $C_k\coloneqq\Z/k\Z$ we have $$\mathcal H(C_k)\cong\Z[u_1,\dots,u_k]\left\langle T\mid (T-u_1)(T-u_2)\cdots(T-u_{k})=0\right\rangle.$$

The specialization of $u_j$ at $\exp\left(j\frac{2i\pi}{k}\right)$ induces a morphism $\mathcal H(C_k)\to \C\otimes\Z[C_k]$.

Moreover, $\mathcal H(C_k)$ is free of rank $k$, with basis $\{1,T,T^2,\dots,T^{k-1}\}$.
\end{thm}
\begin{rmk}
It was remarked by Loïc Poulain-d'Andecy (\cite{lpda}) that, if $R=\Z$, then $\mathcal H_{n,m}(p,q)$ is a specialization of $\mathcal H(C_{n+m})$. Indeed, by Proposition \ref{prop:tor_germ} we have a bijection between their respective basis given by $T\mapsto T_a$ and $T^{n+m-1}\mapsto T_b$. The relation $T_aT_b=p(T_a+T_b)+q$ can then be rewritten as $T^{n+m}=pT^{n+m-1}+pT+q$. Taking a specialization of $(u_1,\dots,u_k)$ at the complex roots of $X^{n+m}-pX^{n+m-1}-pX-q\in \Z[X]$ then induces a specialization $\mathcal H(C_{n+m})\to\mathcal H_{n,m}(p,q)$.
\end{rmk}
\appendix
\section{Finding the correct definition via a diagrammtic approach}\label{sec:hecke_intro}
The first attempts to adapt the definition from Artin--Tits groups to Yang--Baxter structure groups would be to quotient $R[G]$ by something of the form $T_{s^{[d]}}=a_{d-1}T_{s^{[d]}}+\dots+a_1 T_s+a_0$. However, apart from a specific case mentioned in the following sections (the unique non-trivial solution of size 2), this does not really work. Using the GAP package \textit{GBNP} to compute a non-commutative Gröbner Basis, shows that such quotient won't have the correct dimension (it collapses, almost always identifying all generators). 
For instance, the GAP code in Program \ref{gap:dim} checks, for a chosen cycle set of both size and class 3, that no intuitive definition works.
\renewcommand{\thelstlisting}{\arabic{lstlisting}}
\renewcommand{\lstlistingname}{Program}
\lstset{basicstyle={\ttfamily\small}}
\begin{lstlisting}[caption=Checking dimensions of quotient algebras,captionpos=b,frame=tlrb,label=gap:dim]
#Setup
LoadPackage("GBNP");
A:=FreeAssociativeAlgebraWithOne(Integers,"a","b","c");
gens:=GeneratorsOfAlgebra(A);
e:=gens[1];a:=gens[2];b:=gens[3];c:=gens[4];
q:=100;
#Construct all subsets of elements of length < 3
words:=[e,a,b,c,a*a,a*b,b*b,b*c,c*a,c*c];
comb:=Combinations(words);
Remove(comb,1);
sComb:=String(comb);
sComb:=ReplacedString(ReplacedString(
	sComb,"(1)*",""),"<identity ...>","e");
sCombx:=ReplacedString(ReplacedString(
	ReplacedString(sComb,"a","x"),"b","y"),"c","z");
sCombB:=ReplacedString(ReplacedString(
	ReplacedString(sCombx,"x","b"),"y","c"),"z","a");
sCombC:=ReplacedString(ReplacedString(
	ReplacedString(sCombx,"x","c"),"y","a"),"z","b");
combA:=EvalString(sComb);
combB:=EvalString(sCombB);
combC:=EvalString(sCombC);
l:=Length(combA);
#Compute dimensions of each quotient algebras
for i in [1..l] do
Print("\r             ");
Print(i,"/",l);
x:=combA[i];y:=combB[i];z:=combC[i];
rels:=[a*c-b*b,b*a-c*c,c*b-a*a,
	a*b*c-(q-1)*Sum(x)-q*e,b*c*a-(q-1)*Sum(y)-q*e,
	c*a*b-(q-1)*Sum(z)-q*e];
KI:=GP2NPList(rels);
GB:=SGrobner(KI);
if DimQA(GB,0)=27 then
Print("\n");
Print(Sum(x));
Print("\n");
Print(Sum(y));
Print("\n");
Print(Sum(z));
Print("\n");
PrintNPList(GB);
Print("\n");
fi;
od;
\end{lstlisting}
To do this verification for $S=\{s,t,u\},\psi(s)=\psi(t)=\psi(u)=(stu)=\sigma$, we consider all relations of the form $$T_{s^{[d]}}=2T_1+\sum\limits_{\substack{g\in\overline G\\ 1\leq\ell(g)\leq 2}} a_{s,g}T_{g},\quad a_{s,g}\in\{0,1\}\subset\mathbb Q$$
and $T_{t^{[d]}}=\sigma(T_{s^{[d]}})$,$T_{u^{[d]}}=\sigma^2(T_{s^{[d]}})$ to retain the symmetry. Note that we chose a particular specialization of the coefficients $a_i$, as we expect the definition of the Hecke algebra to work for all specializations.
We then use the \textit{GBNP} package functions to compute the size of the quotient algebra (deduced from a non-commutative Gröbner basis). We are interested in quotient algebras which are free of rank $\#\overline G=3^3=27$, so that we can have $(T_g)_{g\in\overline G}$ as a basis. The only relation for which this happen is $T_{s^{[d]}}=2T_1$, i.e. a non-interesting deformation of the group ring $\mathbb Z[\overline G]$. It is also worth to note that, in most cases, the quotient is small to the point that the generators $(T_s)_{s\in S}$ are identified.

This was tested for many small solutions, in particular the cyclic solutions such that $\psi(S)=\sigma\in\Sym_n$, leading to the alternative approach of Section \ref{ss:2gen}. Thus the approach had to be changed, and we are going to give a brief idea on how the current one was obtained. The following approach was inspired by a talk given by L. Poulain d'Andecy in Caen \cite{pda_caen}.

For the Braids groups $B_n$, whose Coxeter groups are $\Sym_n$ (of type $A_{n-1}$), the generic Iwahori--Hecke algebra can be defined by the diagrammatic relations as follows:
\begin{figure}[H]
$$\vcenter{\hbox{\begin{tikzpicture}
\pic[name=b] {braid={s_1 s_1}};
\end{tikzpicture}}}
=
(q-1)\vcenter{\hbox{\begin{tikzpicture}
\pic[name=b] {braid={s_1}};
\end{tikzpicture}}}
+
q\:\vcenter{\hbox{\begin{tikzpicture}
\pic[braid/number of strands=2] {braid={1}};
\end{tikzpicture}}}
$$
\end{figure}

which can also be written as   
\begin{figure}[H]
$$\vcenter{\hbox{\begin{tikzpicture}
\pic[name=b] {braid={s_1}};
\end{tikzpicture}}}
-
q\:\vcenter{\hbox{\begin{tikzpicture}
\pic[name=b] {braid={s_1^{-1}}};
\end{tikzpicture}}}
=
(q-1)\:\vcenter{\hbox{\begin{tikzpicture}
\pic[braid/number of strands=2] {braid={1}};
\end{tikzpicture}}}.
$$
\end{figure}

Intuitively, this means that we are "mostly" interested in the permutation associated to the braid, which is related to the fact that the Coxeter group is $\Sym_n$. In what follows, we will explain the diagrammatical construction which gives the intuition of a "good" definition of Hecke algebra. 

\begin{defi} Let $n$ be a positive integer. Consider the $2n$ points in $\R^2$ with coordinates $(1,0),\dots,(n,0)$, $(1,1),\dots(n,1)$. A family of $n$ curves $(C_i\colon [0,1]\to \R^2)_{1\leq i\leq n}$ is called a $n$-strand permutation diagram if there exists a permutation $\sigma\in\Sym_n$ such that $C_i(0)=(i,1)$ and $C_i(1)=(\sigma(i),0)$. 

In this case, $C_i$ is called the $i$-th strand. 

The inverse of $\sigma$ will be called the permutation associated to the diagram. Equivalently, the associated permutation can be read as the permutation obtained looking at the diagram from bottom to top. 

Two such diagrams are said to be equivalent if they define the same permutation.
\end{defi}
\begin{ex} The following is a $4$-strand permutation diagram with associated permutation $\begin{pmatrix}1&2&3&4\\2&3&4&1\end{pmatrix}=(1234)$:
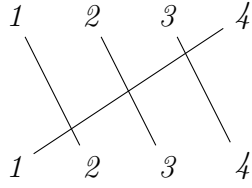
\begin{center}
\begin{tikzpicture}
\node (1s) at (1,2) {1};
\node (2s) at (2,2) {2};
\node (3s) at (3,2) {3};
\node (4s) at (4,2) {4};
\node (1e) at (1,0) {1};
\node (2e) at (2,0) {2};
\node (3e) at (3,0) {3};
\node (4e) at (4,0) {4};
\draw (1s) -- (2e);
\draw (2s) -- (3e);
\draw (3s) -- (4e);
\draw (4s) -- (1e);
\end{tikzpicture}
\end{center}
\end{ex}
If we have two $n$-strand permutation diagrams, we can stack one on top of the other to obtain a new one (after rescaling vertically). This is illustrated in this example:
$$
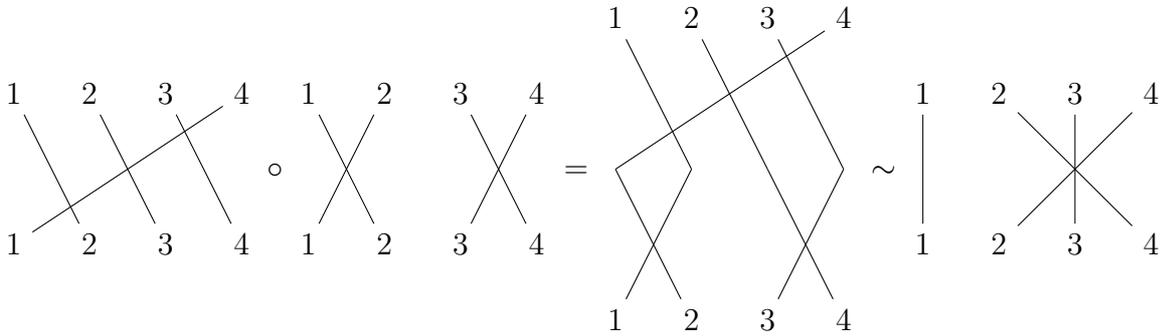
\vcenter{\hbox{\begin{tikzpicture}
\node (1s) at (1,2) {1};
\node (2s) at (2,2) {2};
\node (3s) at (3,2) {3};
\node (4s) at (4,2) {4};
\node (1e) at (1,0) {1};
\node (2e) at (2,0) {2};
\node (3e) at (3,0) {3};
\node (4e) at (4,0) {4};
\draw (1s) -- (2e);
\draw (2s) -- (3e);
\draw (3s) -- (4e);
\draw (4s) -- (1e);
\end{tikzpicture}}}
\circ
\vcenter{\hbox{\begin{tikzpicture}
\node (1s) at (1,2) {1};
\node (2s) at (2,2) {2};
\node (3s) at (3,2) {3};
\node (4s) at (4,2) {4};
\node (1e) at (1,0) {1};
\node (2e) at (2,0) {2};
\node (3e) at (3,0) {3};
\node (4e) at (4,0) {4};
\draw (1s) -- (2e);
\draw (2s) -- (1e);
\draw (3s) -- (4e);
\draw (4s) -- (3e);
\end{tikzpicture}}}
=
\vcenter{\hbox{\begin{tikzpicture}
\node (1s) at (1,2) {1};
\node (2s) at (2,2) {2};
\node (3s) at (3,2) {3};
\node (4s) at (4,2) {4};
\node (1m) at (1,0) {};
\node (2m) at (2,0) {};
\node (3m) at (3,0) {};
\node (4m) at (4,0) {};
\node (1e) at (1,-2) {1};
\node (2e) at (2,-2) {2};
\node (3e) at (3,-2) {3};
\node (4e) at (4,-2) {4};
\draw (1s) -- (2m.center);
\draw (2s) -- (3m.center);
\draw (3s) -- (4m.center);
\draw (4s) -- (1m.center);
\draw (1m.center) -- (2e);
\draw (2m.center) -- (1e);
\draw (3m.center) -- (4e);
\draw (4m.center) -- (3e);
\end{tikzpicture}}}
\sim
\vcenter{\hbox{\begin{tikzpicture}
\node (1s) at (1,2) {1};
\node (2s) at (2,2) {2};
\node (3s) at (3,2) {3};
\node (4s) at (4,2) {4};
\node (1e) at (1,0) {1};
\node (2e) at (2,0) {2};
\node (3e) at (3,0) {3};
\node (4e) at (4,0) {4};
\draw (1s) -- (1e);
\draw (2s) -- (4e);
\draw (3s) -- (3e);
\draw (4s) -- (2e);
\end{tikzpicture}}}
$$
The associated permutation of the first (resp. second) diagram in the product is given by $(1234)^{-1}=(4321)$ (resp. $\left((12)(34)\right)^{-1}=(12)(34)$). And the permutation of their stacking is $(24)^{-1}=(24)$, which is also equal to $(4321)\circ(12)(34)$. The fact that the permutation of the stacking is the product of the permutation holds in general, as indicated by the following:
\begin{prop}\label{prop:diag_sym} There is an isomorphism between the group of $n$-strand permutation diagrams up to equivalence and $\Sym_n$.
\end{prop}
\begin{proof}
Consider the stacking of two diagrams with associated permutations respectively $\sigma$ and $\tau$. The first diagram sends $i$ to $\sigma(i)$, and the second one sends $\sigma(i)$ to $\tau(\sigma(i)$. So we obtain that the permutation of the stacking is the product of the permutation. This implies that, when considering diagrams up to equivalence (defining the same permutation), the stacking operation is a group law: associativity is clear, the identity is the equivalence class of diagrams with trivial permutation, and inverses are given by the equivalence class of diagram with the inverse permutation. In other words, the map sending a diagram to its associated permutation is a morphism. 

Moreover, diagrams are considered up to the equivalence relation of defining the same permutation. Thus there is a unique equivalence class of diagrams with trivial permutation, and so this morphism is an isomorphism.
\end{proof}
\begin{defi} Let $\Gamma$ be a group. A $\Gamma$-marked permutation diagram is a permutation diagram where strands can be marked anywhere by elements of $\Gamma$. There can be multiple ordered elements for one strand. Moreover, a marking by $1\in \Gamma$ is considered equivalent to no marking.

Two markings of one strand are equivalent if they are identified by the group law as follows: 
$$\vcenter
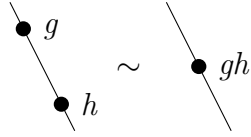
{\hbox{\begin{tikzpicture}
\node (s) at (1,2) {};
\node (e) at (2,0) {};
\node[circle,fill=black,minimum size=0.2cm,inner sep=0pt,label=right:{g}] at (1.25,1.5) {};
\node[circle,fill=black,minimum size=0.2cm,inner sep=0pt,label=right:{h}] at (1.75,0.5) {};
\draw (s) -- (e);
\end{tikzpicture}}}
\sim
\vcenter{\hbox{\begin{tikzpicture}
\node (s) at (1,2) {};
\node (e) at (2,0) {};
\node[circle,fill=black,minimum size=0.2cm,inner sep=0pt,label=right:{gh}] at (1.5,1) {};
\draw (s) -- (e);
\end{tikzpicture}}}$$
\end{defi}
\begin{ex} We will later focus on $\Z$ and $\Z/d\Z$ markings. As those groups are cyclic, we can simplify the markings: 

For $\Z$, associate to $+1$ the marking by $\bullet$ and to $-1$ the marking by $\circ$. A marking by a positive integer $n$ then corresponds to $n$ markings by $\bullet$, and similarly for negative integers with $\circ$. 

For $\Z/d\Z$, we will only consider markings by $\bullet$ which corresponds to the class of $+1$.

The following is a $\Z$-marked $3$-strand permutation diagram, where the strand 1 to 3 are respectively marked by 2, 0 and -3:
$$\vcenter
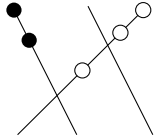
{\hbox{\begin{tikzpicture}
\node (1s) at (1,2) {};
\node (2s) at (2,2) {};
\node (3s) at (3,2) {};
\node (1e) at (1,0) {};
\node (2e) at (2,0) {};
\node (3e) at (3,0) {};
\draw (1s) -- (2e);
\draw (2s) -- (3e);
\draw (3s) -- (1e);
\node[circle,fill=black,minimum size=0.2cm,inner sep=0pt] at (1.1,1.8) {};
\node[circle,fill=black,minimum size=0.2cm,inner sep=0pt] at (1.3,1.4) {};
\node[circle,draw=black,fill=white,minimum size=0.2cm,inner sep=0pt] at (2.8,1.8) {};
\node[circle,draw=black,fill=white,minimum size=0.2cm,inner sep=0pt] at (2.5,1.5) {};
\node[circle,draw=black,fill=white,minimum size=0.2cm,inner sep=0pt] at (2,1) {};
\end{tikzpicture}}}$$
\end{ex}
\begin{rmk}
We can always move all the markings to the top (or bottom) of a strand. This also applies when stacking two diagrams, as illustrated in the following for $\Z/3\Z$-marked $3$-strand permutation diagrams:
$$
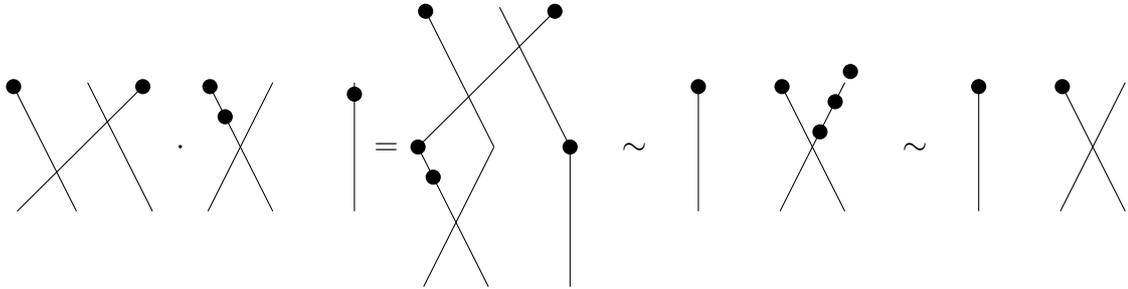
\vcenter{\hbox{\begin{tikzpicture}
\node (1s) at (1,2) {};
\node (2s) at (2,2) {};
\node (3s) at (3,2) {};
\node (1e) at (1,0) {};
\node (2e) at (2,0) {};
\node (3e) at (3,0) {};
\draw (1s) -- (2e);
\draw (2s) -- (3e);
\draw (3s) -- (1e);
\node[circle,fill=black,minimum size=0.2cm,inner sep=0pt] at (2.8,1.8) {};
\node[circle,fill=black,minimum size=0.2cm,inner sep=0pt] at (1.1,1.8) {};
\end{tikzpicture}}}
\cdot
\vcenter{\hbox{\begin{tikzpicture}
\node (1s) at (1,2) {};
\node (2s) at (2,2) {};
\node (3s) at (3,2) {};
\node (1e) at (1,0) {};
\node (2e) at (2,0) {};
\node (3e) at (3,0) {};
\draw (1s) -- (2e);
\draw (2s) -- (1e);
\draw (3s) -- (3e);
\node[circle,fill=black,minimum size=0.2cm,inner sep=0pt] at (1.1,1.8) {};
\node[circle,fill=black,minimum size=0.2cm,inner sep=0pt] at (1.3,1.4) {};
\node[circle,fill=black,minimum size=0.2cm,inner sep=0pt] at (3,1.7) {};
\end{tikzpicture}}}
=
\vcenter{\hbox{\begin{tikzpicture}
\node (1s) at (1,2) {};
\node (2s) at (2,2) {};
\node (3s) at (3,2) {};
\node (1m) at (1,0) {};
\node (2m) at (2,0) {};
\node (3m) at (3,0) {};
\node (1e) at (1,-2) {};
\node (2e) at (2,-2) {};
\node (3e) at (3,-2) {};
\draw (1s) -- (2m.center);
\draw (2s) -- (3m.center);
\draw (3s) -- (1m.center);
\draw (1m.center) -- (2e);
\draw (2m.center) -- (1e);
\draw (3m.center) -- (3e);
\node[circle,fill=black,minimum size=0.2cm,inner sep=0pt] at (1.1,1.8) {};
\node[circle,fill=black,minimum size=0.2cm,inner sep=0pt] at (2.8,1.8) {};
\node[circle,fill=black,minimum size=0.2cm,inner sep=0pt] at (3,0) {};
\node[circle,fill=black,minimum size=0.2cm,inner sep=0pt] at (1,0) {};
\node[circle,fill=black,minimum size=0.2cm,inner sep=0pt] at (1.2,-0.4) {};
\end{tikzpicture}}}
\quad\sim\quad
\vcenter{\hbox{\begin{tikzpicture}
\node (1s) at (1,2) {};
\node (2s) at (2,2) {};
\node (3s) at (3,2) {};
\node (1e) at (1,0) {};
\node (2e) at (2,0) {};
\node (3e) at (3,0) {};
\draw (1s) -- (1e);
\draw (2s) -- (3e);
\draw (3s) -- (2e);
\node[circle,fill=black,minimum size=0.2cm,inner sep=0pt] at (1,1.8) {};
\node[circle,fill=black,minimum size=0.2cm,inner sep=0pt] at (2.1,1.8) {};
\node[circle,fill=black,minimum size=0.2cm,inner sep=0pt] at (3,2) {};
\node[circle,fill=black,minimum size=0.2cm,inner sep=0pt] at (2.8,1.6) {};
\node[circle,fill=black,minimum size=0.2cm,inner sep=0pt] at (2.6,1.2) {};
\end{tikzpicture}}}
\quad\sim\quad
\vcenter{\hbox{\begin{tikzpicture}
\node (1s) at (1,2) {};
\node (2s) at (2,2) {};
\node (3s) at (3,2) {};
\node (1e) at (1,0) {};
\node (2e) at (2,0) {};
\node (3e) at (3,0) {};
\draw (1s) -- (1e);
\draw (2s) -- (3e);
\draw (3s) -- (2e);
\node[circle,fill=black,minimum size=0.2cm,inner sep=0pt] at (1,1.8) {};
\node[circle,fill=black,minimum size=0.2cm,inner sep=0pt] at (2.1,1.8) {};
\end{tikzpicture}}}
$$
where the equality is the stacking operation, the first equivalence is the equivalence of permutation diagram, and the second equivalence is the fact that we have a $\Z/3\Z$-marking (so $\bullet\bullet\bullet=3\bullet=0$)
\end{rmk}

Consider the action of $\Sym_n$ on $G^n$ by permuting the entries, i.e. $\sigma$ sends the $i$-th entry to the $\sigma(i)$-th one, or, equivalently, $\sigma\cdot(g_1,\dots,g_n)=(g_{\sigma^{-1}(1)},\dots,g_{\sigma^{-1}(n)})$ .
\begin{prop}\label{prop:marked_sd} The group of $\Gamma$-marked $n$-strand permutation group is isomorphic to $\Gamma^n\rtimes \Sym_n$, where $\Sym_n$ acts by permuting the entries of $\Gamma^n$.
\end{prop}
\begin{proof}
Let $(g_1,\dots,g_n,\sigma)$ be an element of $\Gamma^n\rtimes\Sym_n$. Consider the map $f$ sending such an element to the permutation diagram associated to $\sigma$ and where the $i$-th strand is marked by $g_i$. 

We have $f\left((g_1,\dots,g_n,\sigma)(h_1,\dots,h_n,\tau)\right)=f(g_1h_{\sigma^{-1}(1)},\dots,g_nh_{\sigma^{-1}(n)},\sigma\tau)$. 

On the other hand, when stacking $f\left((g_1,\dots,g_n,\sigma)\right)$ and $f\left((h_1,\dots,h_n,\tau)\right)$ from bottom to top. The permutation associated to this diagram is $\sigma\tau$ by Proposition \ref{prop:diag_sym}. Moreover, the top diagrams has an $i$-th strand that is followed by the $\sigma^{-1}(i)$-th strand of the second diagram. Thus, the markings on the $i$-th strand of the diagram after stacking is $g_ih_{\sigma^{-1}(i)}$. From this, we deduce that $f$ is a morphism.

Now, $f\left((g_1,\dots,g_n,\sigma)\right)$ is trivial if and only if the diagram has trivial permutation and markings, so $\sigma=\text{id}$ and $g_1=\dots=g_n=1$. This means that $f$ is injective.

Finally, consider a diagram with associated permutation $\sigma$ and markings $g_1,\dots,g_n$. By the definition of $f$, the given diagram is equal to $f(g_1,\dots,g_n,\sigma)$, meaning that f is surjective. Thus $f$ is an isomorphism.
\end{proof}

We can finally arrive at a diagrammatical representation of structure groups and germs of solutions, which corresponds to the I-structure of \cite{etingof,istruct} and Theorem \ref{istruct}.
\begin{thm}
Let $S$ be a cycle set of size $n$ and class $d$. Then its structure group $G$ (resp. germ $\overline G_l$) is isomorphic to a subgroup of $\Z$-marked (resp. $\Z/ld\Z$-marked) $n$-strand permutation diagrams. Moreover, an element is uniquely determined by its marking as a diagram. 
\end{thm}
\begin{proof}
By Theorem \ref{istruct}, we know that $G$ embeds as a subgroup of $\Z^n\rtimes\Sym_n$ such that restricting to the first coordinate is bijective. Theorem \ref{germ} gives a similar embedding of $\overline G_l$ in $(\Z/ld\Z)^n\rtimes\Sym_n$. In both cases, we then apply Proposition \ref{prop:marked_sd} to conclude.
\end{proof}
\begin{rmk}
A way to interpret the quotient $G\to\overline G_l$ through the diagram is to visualize the strands as having thickness in 3-dimensions, and consider the markings as twists. In $G$, a marking as $\bullet=+1\in\Z$ can be seen as a twists by $\frac{2\pi}{ld}$. Then, quotienting to $\overline G_l$ amounts to considering a full twist as trivial.
\end{rmk}

Now going back to the analogy with Artin--Tits group, where the focus to obtain the Iwahori--Hecke algebra was the permutation associated to a braid. Here the permutation of the braid is an obstacle when we only care about the number of circles/twists (the $\Gamma^n$ part). This is why we will consider deformations which only involves elements with trivial permutation. so in our case using $s^{[d]}$. For instance, the analogue of $s^2=(q-1)s+q$ will be ${s^{[d]}}^2=(q-1)s^{[d]}+q$ (where $(s^{[d]})^2=s^{[2d]}$). This means we will consider bigger germs, like here $\overline G_2=G/\langle s^{[2d]}\rangle$ to be able to obtain a Hecke algebra.

The visualization through marked permutation diagrams allows us to understand an important difference between the Garside structures of Artin--Tits groups and Structure groups of solutions to the Yang--Baxter equation. In particular, it yields the intuition on why the "correct" definition will involve elements with trivial permutation.
\emergencystretch=1em
\printbibliography
\end{document}